\documentclass
{amsart}
\usepackage{amssymb,mathrsfs, amsmath, amsfonts}
\usepackage{mathtools}
\usepackage{epsfig}
\usepackage{color}
\usepackage{epstopdf}
\usepackage{graphicx}
\usepackage{srcltx}
\usepackage{amsthm}
\usepackage[mathscr]{eucal}
\usepackage{enumerate}
\usepackage{verbatim}
\usepackage[draft]{todonotes}
\usepackage{relsize}
\usepackage{comment}
\usepackage{hyperref}
\usepackage{empheq}
\usepackage{subfigure}
\usepackage{tikz}
\usepackage{array}
\usepackage{tabularx}
\newcolumntype{C}[1]{>{\centering\arraybackslash}p{#1}}

\usetikzlibrary{arrows,automata}
\usepackage{varwidth}
\usepackage{tikz}
\usetikzlibrary{backgrounds,patterns,calc}
\usepackage{xparse}

\usepackage{enumitem}

\usepackage[normalem]{ulem}
\newcommand{\stkout}[1]{\ifmmode\text{\sout{\ensuremath{#1}}}\else\sout{#1}\fi}

\hypersetup{
	colorlinks=false,       
	linkcolor=blue,          
	citecolor=red,        
	filecolor=magenta,      
	urlcolor=cyan           
}

\newcommand{\Be}{\begin{equation}}
\newcommand{\Ee}{\end{equation}}
\newcommand{\Bea}{\begin{eqnarray}}
\newcommand{\Eea}{\end{eqnarray}}
\newcommand{\Bel}{\begin{align}}
\newcommand{\Eel}{\end{align}}
\newcommand{\Beas}{\begin{eqnarray*}}
	\newcommand{\Eeas}{\end{eqnarray*}}
\newcommand{\Benu}{\begin{enumerate}}
	\newcommand{\Eenu}{\end{enumerate}}
\newcommand{\Bi}{\begin{itemize}}
	\newcommand{\Ei}{\end{itemize}}
\makeatletter
\@namedef{subjclassname@2020}{\textup{2020} Mathematics Subject Classification}
\makeatother

\newcommand{\B}{\Big}

\numberwithin{equation}{section}
\newcommand{\supp} {{\rm supp}\,}
\newcommand{\dist} {{\rm dist}\,}

\theoremstyle{plain}
\newtheorem{thm}{Theorem}[section]
\newtheorem{cor}[thm]{Corollary}
\newtheorem{lem}[thm]{Lemma}
\newtheorem{prop}[thm]{Proposition}

\newtheorem{conj}[thm]{Conjecture}

\theoremstyle{remark}

\theoremstyle{definition}
\newtheorem{defn}[thm]{Definition}

\newcommand{\R}{\mathbb R}
\newcommand{\wa}{e^{it\sqrt{\mathsmaller{-}\Delta}}}

\newcommand{\inv}[1]{\,{#1}^{-1}}

\usepackage{mathtools}

\title[Circular average]
{Circular average relative to  
\\
fractal measures}

\author{Seheon Ham}
\author{Hyerim Ko}
\author{Sanghyuk Lee}

\thanks{}
\address{Department of Mathematical Sciences and RIM, Seoul National University, Seoul 08826, Republic of Korea}
\email{seheonham@snu.ac.kr}
\email{kohr@snu.ac.kr}
\email{shklee@snu.ac.kr}

\begin{document}

\begin{abstract} 
We prove new $L^p$--$L^q$ estimates for averages over dilates of the circle with respect to fractal measures, which unify different types of  maximal estimates for the circular average.   
Our results are consequences of $L^p$--$L^q$ smoothing estimates for the wave operator relative to fractal measures. 
We also discuss similar results concerning the spherical averages. 
\end{abstract}

\maketitle

\section{introduction}\label{Sec:1}
Let $d\ge 2$. We consider the   average 
\[
\mathcal A f(x,t) =  \int_{\mathbb S^{d-1}} f(x- t y) d \sigma(y), \quad x\in \mathbb R^d.
\]
Here $d\sigma $ denotes the normalized measure on the unit sphere $\mathbb S^{d-1}$.  
We study estimates for $\mathcal A$ relative to fractal measures. 
To motivate our study, we briefly review some of the previous results. 

Stein  \cite{Stein}  ($d\ge 3$) and Bourgain \cite{B} ($d=2$)  proved that the maximal operator $f\mapsto \sup_{t>0}|\mathcal Af (\cdot, t)|$ is bounded on $L^p$ 
 if and only if  $\frac d{d-1}< p \le \infty$.    If the supremum is taken over a compact interval $I$ contained in $(0,\infty)$, the consequent maximal function   
has the $L^p$ improving property (\cite{Schlag, SS, Lee}). That is to say, the estimate 
  \begin{equation}\label{av1}
\| \sup_{t \in I} |\mathcal A f(\cdot,t)|\|_{L^q(\mathbb R^d)} \le C_d \| f\|_{L^p(\mathbb R^d)}
\end{equation}
holds for some $p<q$.  Except for a few endpoint cases, we now have a complete characterization of $p,q$ for which  \eqref{av1} holds. 
Let $\mathcal P_d\subset [0,1]\times [0,1]$ be the convex hull of the set consisting of the points  $O:=(0,0),~  Q_1:= \big(\frac{d(d-1)}{d^2+1},\frac{d-1}{d^2+1}\big),$ and $ Q_2:=  \big(\frac{d-1}{d},\frac{1}{d}  \big),~  Q_3:= \big(\frac{d-1}{d},\frac{d-1}{d} \big)$.  
 Schlag \cite{Schlag} ($d = 2$), and  Schlag and Sogge \cite{SS} ($d\ge3$) showed that the estimate \eqref{av1} holds for $(p^{-1}, q^{-1})\in\mathcal P_d\setminus
((O, Q_1)\cup [Q_1, Q_2)\cup [Q_2, Q_3] ) $ and it fails if   $(p^{-1}, q^{-1})\not \in ({ {\mathcal P_d}}\setminus \{ Q_3\})$.
The third named author \cite{Lee} proved the  borderline case $(p^{-1}, q^{-1})\in   (O, Q_1)\cup (Q_1, Q_2)\cup (Q_2, Q_3)$, however the case $(p^{-1}, q^{-1})=Q_1,  Q_2$ remain open. 
See also \cite{RS, BORSS} 
 for recent developments regarding the circular and spherical maximal functions.

Different forms of maximal estimates have also been of interest. In particular,  the estimate
\begin{equation}\label{av2}
\|\sup_{x \in \mathbb R^d} |\mathcal A f(x,\cdot)|\|_{L^q_t ( I)} \le C_d \| f\|_{L^p_\epsilon (\mathbb R^d)} 
\end{equation} 
for some $\epsilon>0$ was used to study  packing problems for the circle and sphere  \cite{Wolff, KW, Mitsis, O}. 
Here 
$\|f\|_{L^p_s(\mathbb R^d)}:=\|(1-\Delta)^\frac s 2 f\|_{L^p(\mathbb R^d)}$.     
When $d\ge3$,  Kolasa and Wolff \cite{KW} showed  \eqref{av2} with {$p=q=2$} for any $\epsilon>0$. 
However, a similar estimate in $\mathbb R^2$ turned out to be more difficult. Wolff \cite{Wolff} obtained  \eqref{av2} with  {$p=q=3$}  for $\epsilon>0$. 
Precisely, results in  \cite{Wolff, KW} are given in a different form in which  averages over the sphere (circle) were replaced by the average 
over annuli of thickness $\delta$. However, it is easy to see that those estimates are equivalent to the abovementioned  estimates in the form of \eqref{av2}. 
 It is also known that \eqref{av2} fails if $\epsilon$ is removed (see   \cite{BR, Kinney} or \cite[Proposition 2.2]{KW} for example). 
Wolff's result was extended to a variable coefficient setting by Zahl \cite{Z2}.

\begin{defn}
For $\alpha \in(0,d+1]$, we say a  non-negative Borel measure $\nu$ on $\R^{d+1}$ is $\alpha$-dimensional  if  there is a constant $C_\nu$
 such that
\begin{equation}\label{measure}
\nu(  \mathbb B^{d+1}(z,\rho) ) \le C_\nu \rho^\alpha, \quad \forall  (z,\rho) \in \mathbb R^{d+1}\times \mathbb R_+,
\end{equation} 
where $\mathbb B^{d+1} (z,\rho)=\{ z'\in \mathbb R^{d+1}: |z-z'|<\rho\}$.      
By $\mathfrak C^{d+1}(\alpha)$ we denote the class of $\alpha$-dimensional measure $\nu$.
For $\nu \in \mathfrak C^{d+1}(\alpha)$, we also define 
\[
\langle \nu\rangle_\alpha := \sup_{z\in \mathbb R^{d+1},\, \rho>0} \rho^{-\alpha} 
\nu(  \mathbb B^{d+1} (z,\rho) ). 
\]

\end{defn}

For the rest of the paper, we fix $I= (1,2)$.  We consider the estimate 
\Be
\label{nu-avr-d}
\| \mathcal Af \|_{L^q ( \mathbb R^d \times I ; d\nu)} \le C\langle\nu\rangle_\alpha^{\frac1q} \|f \|_{L^p (\mathbb R^d)}.
\Ee
Remarkably, the estimate \eqref{nu-avr-d}  unifies the maximal estimates \eqref{av1} and \eqref{av2} in a single framework.  Indeed, it is not difficult to see that  \eqref{nu-avr-d} implies the 
seemingly unrelated estimates \eqref{av1} and \eqref{av2}. One can deduce 
 the estimate \eqref{av1} from \eqref{nu-avr-d} if the latter holds  with a uniform $C$ for all  $d$-dimensional measures with $\langle \nu \rangle_d \le 1$. This can be shown by the Kolmogorov-Seliverstov-Plessner linearization argument and the Riesz representation theorem (see Lemma \ref{sup}).
Similarly, \eqref{av2} follows  if \eqref{nu-avr-d} with $L^p$ replaced by $L^p_\epsilon$  holds uniformly for all $1$-dimensional measures. 

Since the averaging operator $\mathcal A$ is translation-invariant in $x$, it is not possible to have \eqref{nu-avr-d} unless $1\le p\le q \le \infty$.  
Also, taking a specific $\alpha$-dimensional measure, one can show \eqref{nu-avr-d}  holds only if 
$   \alpha p \ge q $ (see  $(i)$ in Proposition \ref{prop:nec}). 
Thus it  is natural for \eqref{nu-avr-d} to assume  $\alpha > 1$  and $p\le q$.   

When $d\ge3$  
the estimate \eqref{nu-avr-d}  is relatively easier to obtain. There are various estimates which are straightforward consequences of the fractal Strichartz estimates for the wave equation (\cite{CHL, Harris2}). 
We discuss the matter in Section \ref{Sec:3}.   However, when $d=2$, such estimates are not enough to prove the estimate \eqref{nu-avr-d}.  
In this paper we focus on the more interesting case $d=2$.

\subsection*{The circular average}	
Let $1<\alpha\le 3$. We define $\mathcal P_2(\alpha) \subset [0,1]^2 $. If $1<\alpha\le 2$,  we set
\[ \mathcal P_2(\alpha) := \Big\{ \Big(\frac 1p, \frac1q\Big)\in [0,1]^2  : \frac 1q\le  \frac 1 p   < \frac \alpha  q,  \   \  \frac{3} p   < 1 + \frac{   \alpha - 1 } q 
\Big\},   \]
and  for $2< \alpha\le 3$,
\[ \mathcal P_2(\alpha) := \Big\{  \Big(\frac 1p, \frac1q\Big)\in [0,1]^2  :
\frac1q\le  \frac 1 p    < \frac \alpha  q, \ \  \frac{3} p   <  1 + \frac{  2\alpha - 3 } q, \ \   \frac 2    p  < 1 + \frac{ \alpha-2 }  q \Big\}. 
\]
In Section \ref{Sec:4}, we show that  \eqref{nu-avr-d} 
fails if $(\inv p, \inv q)\notin \overline{\mathcal P_2(\alpha)} $ when $d=2$.  It seems to be reasonable to conjecture that $\mathcal P_2(\alpha)$ determines, possibly except for the borderline cases,  the optimal range of $p,q$ on which 
\eqref{nu-avr-d} holds.\footnote{
As seen in the above, some of the borderline cases are actually not true.}

 \begin{conj} 
\label{conj-circle}
Let $d=2$ and $1< \alpha \le 3$. 
The estimate \eqref{nu-avr-d} 
holds for $\nu \in \mathfrak C^3(\alpha)$ if $(\inv p, \inv q) \in  \mathcal P_2(\alpha)$. 
\end{conj}

The following is our main result,  which verifies  the conjecture 
for $\alpha\in [3- \sqrt 3,  3]$. 

\begin{thm}\label{thm:circular-avr}
Let $d=2$, $1< \alpha \le 3$, and $\nu \in \mathfrak C^3(\alpha)$.   
Then, the estimate \eqref{nu-avr-d}  holds if 
\begin{equation}\label{pq}
(\inv p, \inv q) \in  \mathcal P_2(\alpha) ~\text{ and }~    
p   >(6-2\alpha)/ \alpha. 
\end{equation} 
\end{thm}

Note that $4-\alpha \ge(6-2\alpha)/ \alpha$ if $\alpha\ge 3- \sqrt 3$ and $p>4-\alpha$ whenever $(1/p, 1/q)\in \mathcal P_2(\alpha)$. 
As a corollary we obtain $L^p(\mathbb R^2)$--$L^q(d\mu)$ estimate for the circular maximal function relative to  fractal measures.
\begin{cor}\label{thm:circular}
Let $1< \alpha \le 2$ and $\mu \in \mathfrak C^2(\alpha)$. 
For $p,q$ satisfying \eqref{pq},  we have
\begin{equation}\label{est-circ}
\big\|   \sup_{t \in I} |\mathcal A  f (\cdot,t)| \big\|_{L^q(d\mu)}   \le C \langle \mu\rangle_\alpha^{\frac 1q} \|f \|_{L^p(\mathbb R^2)} .
\end{equation}
\end{cor}

A modification of the examples which give the necessary conditions of \eqref{nu-avr-d} (Proposition \ref{prop:nec})  shows that \eqref{est-circ} holds only if 
$(\inv p, \inv q) \in  \overline{\mathcal P_2(\alpha)}$. 
Consequently, Corollary \ref{thm:circular} establishes  boundedness on sharp range for $\alpha\in (3- \sqrt 3,  2]$.
 
 $L^p(\mathbb R^d)$--$L^q(d\mu)$ estimate for $\mathfrak Mf:=\sup_{t \in I} |\mathcal A  f(\cdot, t) | $  has been utilized
 to study problems in geometric measure theory.   
$L^2(\mathbb R^d)$--$L^2(d\mu)$ estimate was studied by Mitsis \cite{Mitsis} for  $1<\alpha\le d$,  $d\ge3$.
 D. Oberlin and R. Oberlin \cite{OO} obtained 
$L^{2}_{(1-\alpha)/2+}(\mathbb R^d) $--$ L^q(d\mu)$ bound on $\mathfrak Mf$ when $1<q<2$ and $0<\alpha <(d-1)/2$. 
Also,  for $d\ge3$, $L^p(d\mu_1)$--$L^p(d\mu_2)$ estimate was discussed 
in Iosevich et al. \cite{IKSTU} with  $\alpha_i$-dimensional measure $\mu_i$,  $1<\alpha_i \le d$, $i=1,2$. Those estimates were used to study the sphere packing problem, Hausdorff  dimension of the pinned distance set, and the divergence set of the solution to the wave equation (see also \cite{HKL}).    

Compared with the previous results,  it is remarkable that  
Corollary \ref{thm:circular} establishes the sharp $L^p$ improving property for the circular maximal function relative to some $\alpha$-dimensional measures. As far as the authors  are aware, no such results have appeared before. In \cite{IKSTU} non-sharp local smoothing estimates were combined with  a duality argument. However,  the estimates in Corollary \ref{thm:circular} can not be obtained
 by the approach in \cite{IKSTU} even if combined with the sharp local smoothing estimate \cite{GWZ}.   Nevertheless, such estimates are easier to prove (see Corollary \ref{3dm}) in higher dimensions.

Furthermore, one can use  the  estimate \eqref{est-circ} (with $p=q$) to prove the circle packing theorem which was shown in \cite{Wolff3}:  
 \emph{Suppose $F\subset \mathbb R^2$ is a Borel set of Hausdorff dimension exceeding $1$ and $E$ is  a compact  set 
 such  that $E\subset \mathbb R^2$ contains circles centered at each point in $F$. Then $E$ has positive Lebesgue measure.} 
 
Corollary \ref{thm:circular} with $\alpha=2$ recovers the  results on the circular maximal function except for the endpoint cases  (\cite{Schlag, SS, Lee}).

\subsubsection*{Local smoothing estimate relative to $ {\nu}
\in \mathfrak C^3(\alpha)$.}
In order to prove Theorem \ref{thm:circular-avr},  we  consider  $L^p_\gamma$--$L^q(d\nu)$  estimate    for the wave operator
\[  	e^{ i t \sqrt{-\Delta}}  f(x)  = \int e^{i(x\cdot \xi+t |\xi|)} \widehat f(\xi)\,d\xi  . \]
For $2<\alpha \le 3$,  we  obtain the estimate \eqref{nu-avr-d}  
by relying  mainly on $L_\gamma^2$--$L^q(d\nu)$ estimate for $e^{ i t \sqrt{-\Delta}}$, which is to be discussed in Section \ref{Sec:3}.
However,  for $1<\alpha \le 2$, such $L_\gamma^2$--$L^q(d\nu)$ estimates are not enough for our purpose. 
Instead,  we exploit estimates for the wave operator  
relative to $\alpha$-dimensional measure $\nu$:
\begin{equation}\label{LS-AAA}
	\B(   \int_{\mathbb B^2(0,1)\times I} \big|\wa f(x)\big|^q   d\nu(x,t) \B)^{1/q} \le 
	C \langle \nu\rangle_\alpha^{\frac1q}  \|f \|_{L^p_{\gamma}  (\mathbb R^2)}  .	\end{equation}
Sogge  conjectured that \eqref{LS-AAA} holds with  $d\nu(x,t)=dxdt$ 
for $ p=q\ge {2d}/(d-1)$ if $\gamma>  (d-1)(\frac12 -\frac1p)-\frac1p$. 
The  conjecture was studied by various authors \cite{Wolff3, LV,  BD}.
Recently,  Guth, Wang, and Zhang \cite{GWZ} confirmed the conjecture in two dimensions while it remains open in higher dimensions. 
See also \cite{MSS, SS}  for earlier results.

For $1<\alpha \le 2$,  Theorem \ref{thm:circular-avr} follows from the next theorem, which we prove by combining  the sharp $L^p$ local smoothing estimate \cite{GWZ}  and the trilinear 
restriction estimate  for the cone (see \cite{BCT, BG,LV}).  

\begin{thm}\label{smoothing}
Let $1<\alpha \le 2$ and $\nu \in \mathfrak C^3(\alpha)$. If $p,q \ge1$ satisfy \eqref{pq},  then  the estimate   \eqref{LS-AAA} holds  
for some $\gamma<1/2$. 
	\end{thm}

For the proof of Theorem \ref{thm:circular-avr}  we need only to obtain  the estimate \eqref{LS-AAA} for $\gamma<1/2$ (see Section \ref{Sec:3}). 
However,  it is of independent interest  to characterize $p,q$ and $\gamma$ for which  \eqref{LS-AAA} holds. 
In this regard, we can obtain the sharp  estimate on a  restricted range of $p,q$.  To state our result, we set 
\begin{equation}\label{ka}
\kappa (\alpha)  = \begin{cases}
\, 3 ,~ & ~   \hfill 2<\alpha \le 3,\\
\, \alpha+1 , ~ & ~   \hfill 1<\alpha \le 2, \\
\,  2\alpha, 
~ & ~   \hfill 0<\alpha \le 1.
\end{cases}
\end{equation}

\begin{thm}\label{LS-est} 
		Let $0<\alpha \le 3$ and  $\nu \in \mathfrak C^3(\alpha)$. 
	 If $
	\frac1q \le \min\big( \frac1p , ~\frac1{3p}+\frac16,~ \frac2{3p'} \big) $ and $\frac1p +    \frac{\kappa(\alpha)}{q}  \le 1 
	$, then \eqref{LS-AAA} 
	holds for $\gamma >\frac12 +\frac1p -\frac{\alpha}{q}$. 
\end{thm}

The regularity  assumption $\gamma> \frac 12+\frac 1p-\frac \alpha q$ is sharp  
in that   \eqref{LS-AAA} fails in general if $\gamma< \frac 12+\frac 1p-\frac \alpha q$  (see Section \ref{Sec:2}).
It seems to be natural to conjecture that  the same remains valid provided that $\frac 1q \le \frac1p$ and $\frac1p + \frac{\kappa(\alpha)}{q}  \le 1 $. 
This in turn implies Conjecture \ref{conj-circle}.  When $2<\alpha \le 3$, $\frac1q\le \frac1p$, and $\frac1p +\frac3q\le 1$, the estimate  \eqref{LS-AAA} is a consequence of  the sharp local smoothing estimate \cite{GWZ}. 
In higher dimensions ($d\ge 3$),  the local smoothing
estimate in \cite{BD} and Lemma \ref{elementary} yield  an analogue of \eqref{LS-AAA}  for $d\le \alpha\le d+1$
whenever  $\gamma> \frac{d-1}{2}  +  \frac{1}{p} -\frac \alpha q $, $\frac1p +\frac{d+3}{d-1}\frac1q \le 1$, and  {$p\le q$}.

\subsubsection*{Estimate for the spherical average}
In higher dimensions  one can obtain results similar to Theorem \ref{thm:circular-avr} for the spherical average in $\mathbb R^{d+1}$, $d\ge 3$ (see Theorem \ref{maxthm-in-general}).
Such results can be shown by making use of  the Strichartz estimates for the wave equation relative to fractal measure (see Theorem \ref{3d-frac}).
The spherical average estimates in Theorem \ref{maxthm-in-general} seem to be sharp for $d=3$, while we only manage to verify the optimality  for integer $\alpha$. See Section \ref{Sec:4}.

\subsubsection*{Organization of the paper}
In Section \ref{Sec:2}, we obtain the local smoothing estimates relative to $\alpha$-dimensional measure (Proposition \ref{trg} and Proposition \ref{trg22}), which we use to prove Theorem \ref{smoothing} and Theorem  \ref{LS-est}. 
 In Section \ref{Sec:3} we prove  Theorem \ref{smoothing},  Theorem \ref{LS-est}, Theorem \ref{thm:circular-avr}, and Corollary \ref{thm:circular}, and obtain the estimates in higher dimensions.  
We discuss the sharpness of the estimate \eqref{nu-avr-d}  in Section \ref{Sec:4}.

\section{Weighted local smoothing estimates}
\label{Sec:2}

In this section, we obtain some preparatory results which we use in proving Theorem \ref{smoothing} and Theorem \ref{LS-est}. 
The main object  is to prove Proposition \ref{trg} and Proposition \ref{trg22}.
To this end, we adapt the argument   due to Vargas and one of the authors \cite{LV} which is based on  the trilinear restriction estimate. 

We decompose $e^{it\sqrt{-\Delta} } f$ into two parts: a sum of operators with small frequency supports and a product of operators with angularly separated frequency supports.
For the latter we apply the multilinear restriction estimate due to Bennett, Carbery, and Tao \cite{BCT} by making use of the fact that angular separation implies transversality on the trilinear operator.  
For the first part, we rescale the operator and apply an induction argument (see \eqref{QR}).  
However, the support of a rescaled measure is no longer contained in a bounded set, so the induction can not be carried out on a bounded set.
To get around this issue, we consider a class of weight functions which are not necessarily supported in a bounded set.

\subsection{Weighted local smoothing estimates in $\mathbb R^{2+1}$}

For $0<\alpha \le 3$, we denote by  $\Omega(\alpha)$ the class  of nonnegative measurable functions $\omega : \mathbb R^2\times I \mapsto \mathbb R$ which satisfy 
\Be
\label{omega}
\int_{\mathbb B^3(z, r)} \omega(y) dy \le  r^\alpha,  \quad \forall  (z,r) \in \mathbb R^3\times  \mathbb R_+. 
\Ee
By Littlewood-Paley decomposition, 
it is enough to consider the estimate   
\begin{equation}\label{TRg}
\| e^{it\sqrt{-\Delta}} P_\lambda f \|_{L^q(\mathbb R^2\times I;\omega )} \le C \lambda^{\gamma} \|f\|_{p}  .
\end{equation} 
Here $P_\lambda$ is the standard Littlewood-Paley projection operator given by $\widehat{P_\lambda f}(\xi) = \beta(\lambda^{-1}|\xi|)\widehat f(\xi)$ for $\beta \in \mathrm C_c^\infty((1/2,2))$ and $\| f\|_{L^q(E; \omega)}^q :=   \int_E |f(x)|^q \omega(x)dx $.

Instead of proving \eqref{TRg} directly, we consider a modified operator to facilitate rescaling. We consider 
\[
\psi_{\mathsmaller{0}}(\theta) := \sqrt{1+\theta^2} -1.
\]
Note that $\psi_{\mathsmaller{0}}$ can be regarded  as a small perturbation of $\theta^2/2$.

For $\lambda\ge 1$ and $J \subset [-1,1]$, let  
\[ A_\lambda(J) =  \{  (\eta, \rho) : \lambda/4 \le \eta \le 4\lambda,~ \theta:=\rho/\eta \in J \}.\]
We define 
\begin{equation}\label{TR}
T_\lambda^\phi  g (x,t) = \int_{A_\lambda([-1,1])} e^{ i (x,t) \cdot (\eta,\rho,  \eta \phi(\rho/\eta))} \beta (\lambda^{-1}\eta  )\beta_0 ( |\rho|/\eta ) 
\widehat g(\eta,\rho) d\eta d\rho  ,  
\end{equation} 
where  $\beta_0$ is a smooth function supported in $[0,2]$.

\begin{prop}\label{trg}
Let  $1<\alpha\le2$ and $\omega \in \Omega(\alpha)$.
If $p,q \ge1$ satisfy \eqref{pq} and $\lambda \gg 1$, then for some $\gamma<1/2$ we have
\begin{equation}\label{TR-est}
\| T_\lambda^{\psi_0} g \|_{L^q(\mathbb R^2\times I; \omega )} \le C \lambda^{\gamma} \|g\|_{p}.
\end{equation}
\end{prop} 

We deduce the estimate \eqref{TRg} from \eqref{TR-est} via a change of variable (see the proof of Theorem \ref{smoothing} in Section \ref{Sec:3}).

\subsubsection*{Normalization of the phase function.} For our induction  argument, we begin with normalizing the phase function of $T_\lambda g$. 
Let \[ \phi_\circ(\theta)=\frac 12\theta^2.\]
For a given $\epsilon>0$, we define a class $ \mathfrak S(\epsilon)$  of $\mathrm C^3$ functions by 
\[
\mathfrak S(\epsilon) = \{ \phi \in \mathrm C^3([-1,1]) : \phi(0)=\phi'(0)=0,~ \|    \phi  - \phi_\circ  \|_{\mathrm C^3([-1,1])} \le \epsilon   \} ,
\]
where $\|    \phi   \|_{\mathrm C^n([-1,1])}   = \sum_{j\le n}\sup_{\theta\in[-1,1]} \big| \partial_\theta^j   \phi( \theta)    \big| $.
For $\theta_\circ\in (-1,1) $, let   $h$ be a real number such that $[\theta_\circ-h,\theta_\circ+h] \subset [-1,1]$. 
For $\phi\in \mathrm C^3([-1,1]) $ satisfying $\phi''(\theta_\circ) \neq 0$, we define  
\[
\phi_{\theta_\circ,h}(\theta) =(\phi''(\theta_\circ)h^2) ^{-1}\big( \phi(h\theta+\theta_\circ)-\phi(\theta_\circ)-\phi'(\theta_\circ)h\theta \big) .
\]

\begin{lem}\label{para}
Suppose $\phi \in \mathfrak S(\epsilon_\circ)$ for  some  {$0<\epsilon_\circ  < 1/8$}. Then,  there exists $0<h_\circ=h_\circ(\epsilon_\circ)$ such that $\phi_{\theta_\circ,h} \in \mathfrak S(\epsilon_\circ)$ whenever  $0<h\le h_\circ$ and $[\theta_\circ-h,\theta_\circ+h]\subset [-1,1]$. 
\end{lem}

Throughout the paper, $\epsilon_\circ$ denotes a positive constant such that  $0<\epsilon_\circ  < 1/8$.

\begin{proof}
By Taylor series expansion, we have
\begin{align*}
\phi( h\theta+\theta_\circ) -  \phi(\theta_\circ) -  \phi'(\theta_\circ)h\theta   = \frac{\phi''(\theta_\circ) h^2\theta^2}{2}  + \mathcal E(\theta_\circ, h,\theta)  ,
\end{align*}
where  $ \|   \mathcal E(\theta_\circ, h,\cdot) \|_{\mathrm C^3([-1,1])} \le C  h^3$  with a constant $C$ independent of $\theta_\circ\in [-1,1]$. 
Since  $\phi \in \mathfrak S(\epsilon_\circ)$,  $\phi''(\theta_\circ) \in [3/4, 5/4]$.  
So, we see
\[
\|  \phi_{\theta_\circ,h}  - \phi_\circ \|_{\mathrm C^3([-1,1])} \le (\phi''(\theta_\circ)h^2)^{-1}  \|\mathcal E(\theta_\circ, h,\cdot) \|_{\mathrm C^3([-1,1])} \le Ch . 
\]
Thus,  there exists $h_\circ \in (0, \epsilon_\circ]$ such that  $\phi_{\theta_\circ,h} \in \mathfrak S(\epsilon_\circ)$
for  $0<h\le h_\circ$.
\end{proof}

\subsubsection*{Induction quantity.} As mentioned before, our proof is based on an induction on scale argument. 
We introduce a quantity which makes it possible to carry out the induction argument.

For $1 \le p\le q \le \infty$, we define $Q(\lambda) = Q(\lambda, \epsilon_\circ, p,q,\alpha)$ by
\begin{equation}\label{QR}
\begin{split}
Q(\lambda) 
=\sup \Big\{ \| T_\lambda^\phi g  \|_{L^q(\mathbb R^2\times I;\omega)}  :  \  &  \supp {\widehat g} \subset A_\lambda([-1,1]), ~\| g \|_{p}\le 1 ,  \\
\ &   \phi \in \mathfrak S(\epsilon_\circ),~ \omega \in \Omega (\alpha) \Big\} .  
\end{split}
\end{equation}
It is easy to see $Q(\lambda)<\infty$. Indeed, 
one can show a rough bound 
\begin{equation}\label{q-lambda}
Q(\lambda)\lesssim \lambda^2,  \quad   \lambda \ge 1
\end{equation}
for $1 \le p\le q \le \infty$. 
We write $T_\lambda^\phi g(x,t) = K_\lambda(\cdot,t)\ast g$, where the kernel is given by 
\begin{align}\label{kernelA}
K_\lambda (x,t )=\int e^{i(x ,t) \cdot (\eta, \rho, \eta \phi(\rho/\eta) ) } \beta (\lambda^{-1}\eta) \beta_0(|\rho/\eta| ) d\eta d\rho.
\end{align}
In order to show \eqref{q-lambda}, we use an easy  estimate:
  \begin{equation}\label{kernelB}
|K_\lambda(x,t)| \lesssim \mathcal   K(x):= \lambda^2  \big(1+   |x|     \big)^{-M} , \quad  t \in I
\end{equation} 
for any $M\ge1$,
which follows by routine integration by parts (for example, see  \eqref{ker-wave} below).
Using  \eqref{omega}, it is easy to see 
\begin{align}\label{claim01}
\Big\|\int  \mathcal K(x-y)g(y)\,dy \Big\|_{L^q(\mathbb R^2\times I,\omega)} \lesssim \lambda^2 \|g\|_{p}  
\end{align} 
for $1\le p \le q$.
Indeed, since  $\int_I \int  \mathcal K(x-y)\,\omega(x,t)\,dxdt \lesssim \lambda^2$
and $\int  \mathcal K(x-y) dy \lesssim \lambda^2$,
 \eqref{claim01} follows for $p=q$ by Schur's test.
Also, by H\"older's inequality, we get \eqref{claim01} for $q=\infty$ and $1 \le p \le \infty$. 
Interpolating those estimates, we have \eqref{claim01}  and therefore \eqref{q-lambda} for $1\le p \le q$.

\medskip

In what follows, we show $Q(\lambda) \le 
 C_\epsilon  \lambda^{\frac12 - \epsilon}$ for some  $\epsilon>0$.

\subsubsection*{Rescaling}
We first observe that $L^p$--$L^q$  bound on  $T_\lambda^\phi$ is improved on a certain range of $p,q$ if $\widehat g$ is supported in a narrow conic neighborhood.   
More precisely,  we have the following.

\begin{lem} \label{Lscale} 
Let $0<\alpha \le 3$, $\omega \in \Omega(\alpha)$, and $\phi \in \mathfrak S(\epsilon_\circ)$.
For $h \ge \lambda^{-\frac 12}$, let  $J_h=[\theta_\circ- h , \theta_\circ+ h ] \subset [-1,1]$.
Suppose $\supp \widehat{g} \subset A_\lambda(J_h)$.  
Then,   
there is  an $h_\circ=h_\circ (\epsilon_\circ)$ such that      
\begin{equation}\label{scale}
\| T_\lambda^\phi g\|_{L^q(\mathbb R^2 \times I; \omega)}
\le C    h^{\frac{2\alpha - \kappa(\alpha) }q -\frac 3p }
Q(\phi''(\theta_\circ)h^2 \lambda) \|g\|_{p}   
\end{equation}
whenever $h\le h_\circ$. 
\end{lem}

We note $\kappa(\alpha) = \min \{ 3, \alpha+1, 2\alpha \}$ (see \eqref{ka}).

\begin{proof}
Since $\widehat g(\eta,\rho)$ is supported in $A_\lambda(J_h)$, we have $\lambda/4 \le \eta \le 4\lambda$ and $\theta  = \rho/\eta \in J_h$ i.e., $|\theta - \theta_\circ|\le h$.
We set \[  
D_{\theta_\circ,h} = \frac1{\phi''(\theta_\circ)}
\begin{bmatrix} 
h^{-2} & h^{-2}\theta_\circ & h^{-2}\phi(\theta_\circ) \\ 
0 & h^{-1} & h^{-1}\phi'(\theta_\circ) \\
0& 0& \phi''(\theta_\circ)
\end{bmatrix} ,     
\]
and
\[ 
L_{\theta_\circ,h}(\eta,\rho) = \frac1{\phi''(\theta_\circ)}(h^{-2} \eta, h^{-1} \rho+h^{-2}\theta_\circ \eta).
\]

Let $z = (x,t)$. Note that 
$z\cdot (\eta, \rho, \eta \phi(\rho/\eta) )\mapsto D_{\theta_\circ,h}\,z\cdot  (\eta, \rho, \eta \phi_{\theta_\circ,h}(\rho/\eta) )$ under the transformation  $(\eta, \rho) \mapsto L_{\theta_\circ,h}(\eta,\rho) $. 
Changing variables $(\eta, \rho) \mapsto L_{\theta_\circ,h}(\eta,\rho)$, we obtain
\[ 
T_\lambda^\phi g(z )
= T_{\lambda_1}^{\phi_{\theta_\circ,h}} g_{\theta_\circ,h} (D_{\theta_\circ,h}z), \qquad \lambda_1=\phi''(\theta_\circ)h^2\lambda ,
\] 
where
\[ \widehat{g_{\theta_\circ,h}}(\eta,\rho)  =  {\det (L_{\theta_\circ,h})} \widehat g(L_{\theta_\circ,h}(\eta,\rho)).
\] 
Clearly, $\widehat{g_{\theta_\circ,h}}$ is supported in $A_{\lambda_1}
([-1,1])$.   By Lemma \ref{para}
 there exists $h_\circ=h_\circ(\epsilon_\circ)$ such that 
$\phi_{\theta_\circ,h} \in \mathfrak S(\epsilon_\circ)$ for $h \le h_\circ$.

Let us define \[  \omega_{\theta_\circ,h} (z) =  (1+ \widetilde C)^{-1}
h^{ \kappa(\alpha)-2\alpha+3}  \omega(D_{\theta_\circ,h}^{-1} z),  
\] 
where   $\widetilde C>0$  is a constant  to be chosen later. We claim $\omega_{\theta_\circ,h}  \in \Omega (\alpha)$ if $\widetilde C>0$ is large enough.   
To show this, we first observe
 \[ \int_{\mathbb B^3(y_\circ, r)} \omega(D_{\theta_\circ,h}^{-1}z) dz=  h^{-3}  
 \int_{D_{\theta_\circ,h}^{-1}\mathbb B^3(y_\circ, r)} \omega(z) dz.\]
So,  it is sufficient to show 
\begin{align}
\label{omega2}
  \int_{D_{\theta_\circ,h}^{-1}\mathbb B^3(y_\circ, r)} \omega(z) dz
\le   C r^\alpha h^{2\alpha - \kappa(\alpha)}.
\end{align}
Note that $D_{\theta_\circ,h}^{-1} \mathbb B^3(y_\circ, r)$ is contained in a  box 
of dimensions
$c  h^2 r  \times c hr   \times c  r$ for a constant $c>0$.  We can cover  $D_{\theta_\circ,h}^{-1} \mathbb B^3(y_\circ, r)$ 
with  as many as $Ch^{-3}$, $Ch^{-1}$, and $C$  balls  of radii $h^2 r$, $hr$, and $r$, respectively. By \eqref{omega}
the left hand side of \eqref{omega2} is bounded above by $C \min\{ r^\alpha h^{2\alpha-3}, r^\alpha h^{\alpha-1}, r^\alpha\}=C  r^\alpha h^{2\alpha -\min \{ 3, \alpha+1, 2\alpha \}}$ for some $C$. 

Changing variables $z \mapsto D_{\theta_\circ,h}^{-1} z$,  
we have
\begin{align}\label{T-rescale}
 \|T_\lambda^\phi g \|_{L^q(\mathbb R^2\times I;\omega)}^q  
&  =   (1+\widetilde C)  h^{2\alpha - \kappa(\alpha)  } \int_{\mathbb R^2\times I }  
\big|  T_{\lambda_1}
^{\phi_{\theta_\circ,h}}   g_{\theta_\circ,h} (  z ) \big|^q \omega_{\theta_\circ,h}(z) dz.  
\end{align}
Since  $\phi_{\theta_\circ,h} \in \mathfrak S(\epsilon_\circ)$, $\omega_{\theta_\circ,h}  \in \Omega (\alpha)$,  and $\widehat{g_{\theta_\circ,h}}$ is supported in $A_{\lambda_1}
([-1,1])$, by the definition of $Q$ we have 
\[
\int_{\mathbb R^2\times I }  
\big|  T_{\lambda_1}
^{\phi_{\theta_\circ,h}}   g_{\theta_\circ,h} (  z ) \big|^q \omega_{\theta_\circ,h}(z) dz\le \big(  Q(\phi''(\theta_\circ)h^2 \lambda) \|g_{\theta_\circ,h}\|_{p} \big)^q.
\] 
Therefore, we get \eqref{scale} since $\|g_{\theta_\circ,h}\|_{p} = h^{-\frac 3p} \|g\|_{p}$.
\end{proof}

\subsection{Trilinear estimate.}
In this section, we obtain  a weighted local $L^p$--$L^{q}$ estimate for the trilinear operator $\prod_{i=1}^3 T_\lambda^\phi g_i$ by interpolating the multilinear restriction estimate due to Bennett, Carbery and Tao \cite{BCT}, the local smoothing estimate in \cite{GWZ}, and an easy $L^1$--$L^\infty$ estimate. 
Subsequently, we extend the local estimate to a global one by using decay of the kernel $K_\lambda$.
For simplicity, we write \[T_\lambda = T_\lambda^\phi.\]

For  {$\phi \in \mathfrak S(\epsilon_\circ)$} 
and an interval $J \subset [-1,1]$, we consider a truncated conic surface  
\[
  \Gamma_\lambda (J) := \{ (\eta , \rho, \tau )  \in \mathbb R^3:  \tau = \eta\phi(\rho/\eta),~ (\eta,\rho) \in A_\lambda(J) \} .
\]
Let $\mathrm N(\theta)$ be the normal vector to $\Gamma_\lambda(J)$ at $\eta(1 , \theta,\phi(\theta ))$. A computation gives
\Be 
\label{normal} \mathrm  N(\theta)   = (\phi(\theta)-\theta \phi'(\theta), \phi'(\theta), -1).
\Ee
The following is a consequence of  the multilinear restriction estimate due to Bennett-Carbery-Tao \cite{BCT}. 
We denote by $d\sigma_\lambda$  the surface measure on  $\Gamma_\lambda([-1,1])$.
\begin{thm}
\label{BCT}
Let $\phi \in \mathfrak S(\epsilon_\circ)$ and  $J_1, J_2, J_3\subset [-1,1]$ be intervals. 
Suppose  
\begin{equation} \label{transv}
\det (\mathrm  N(\theta_1),\mathrm  N(\theta_2),\mathrm N(\theta_3)) \ge \delta
\end{equation}
holds for some $\delta >0$  whenever $\theta_i \in J_i$, $i=1,2,3$. 
If $\epsilon_\circ$ is sufficiently small and $\lambda \gg \delta^{-1}$, then for $\epsilon>0$
\[ 
\Big\| \prod_{i=1}^3  \widehat{f_i d\sigma_\lambda}  \Big\|_{L^1(\mathbb B^3(0,10))}
\le C_\epsilon \lambda^{\epsilon }
\prod_{i=1}^3  \|f_i\|_{L^2(\Gamma_\lambda(J_i))}
\] 
whenever $\supp f_i \subset \Gamma_\lambda (J_i)$, $i=1,2,3$.
\end{thm}

By interpolation with the local smoothing estimate and an easy $L^1-L^\infty$ estimate, we get the following.

\begin{prop}\label{localtri}
Let $\phi \in \mathfrak S(\epsilon_\circ)$, and  let $J_1, J_2, J_3\subset [-1,1]$ be intervals. 
Also let $\frac 1q \le \min(\frac 1p,~ \frac{1}{3p}+\frac 16,~ \frac 2{3p'})$ for $ 1 \le p \le \infty$.
Suppose \eqref{transv} holds whenever $\theta_i \in J_i$, $i=1,2,3$.  
If $\epsilon_\circ$ is sufficiently small and $\lambda  \gg \delta^{-1} $, 
then we have 
\begin{equation}\label{triL}
\Big\| \prod_{i=1}^3  T_\lambda g_i  \Big\|_{L^{\frac q3}(\mathbb B^2(0,1)\times I)}
\le C  \lambda^{3\gamma_0}
\prod_{i=1}^3 \| g_i \|_{p}, 
\end{equation}
for $\gamma_0 >\frac12 +\frac1p-\frac3q$ whenever $\supp \widehat {g_i} \subset A_\lambda( J_i)$, $i=1,2,3$. 
\end{prop}
\begin{proof}
By interpolation it is enough to show \eqref{triL} for $(\frac 1p, \frac 1q)=(0,0)$, $(\frac 14,\frac14)$, $(\frac12,\frac13)$, and $(1,0)$ if $\gamma_0>\frac 12+\frac 1p-\frac 3q$. 
For the first two cases, \eqref{triL} follows from the local smoothing estimate   \eqref{LS-AAA} with $d\nu=dxdt$ which holds  for $4\le p=q \le \infty$ and $\gamma_0 >\frac12 - \frac2p$.  
By Theorem \ref{BCT} and Plancherel's theorem,  we have  \eqref{triL} with   $(\frac 1p, \frac 1q)=(\frac12,\frac13)$ for $\gamma_0>0$. 
Also,  \eqref{kernelA} and van der Corput's lemma yield   $\|K_\lambda \|_{\infty} \le \lambda^{3/2}$ (e.g.  \eqref{ker-wave}), which gives 
$\|T_\lambda g_i \|_{L^{\infty}(\mathbb B^2(0,1)\times I)} \lesssim   \lambda^{3/2} \| g_i \|_{1}$.  Thus, we have    \eqref{triL} with   $(\frac 1p, \frac 1q) = (1,0)$ for $\gamma_0 \ge 3/2$. 
\end{proof}

From \eqref{kernelB}, we see that $ K_\lambda(\cdot, t) $ decays rapidly away from $\mathbb B^2(0,1)$. 
Using this, we can extend the local   estimate \eqref{triL} to a global one.  

\begin{lem}\label{globtri}
Under the same assumptions as in Proposition \ref{localtri}, we have 
\begin{equation}\label{pqest}
\Big\| \prod_{i=1}^3  T_\lambda g_i \Big\|_{L^{\frac q 3} (\mathbb R^2\times I)}
\le C \lambda^{3\gamma_0 } 
\prod_{i=1}^3 \| g_i \|_p   
\end{equation}
for $\gamma_0>\frac 12+\frac 1p-\frac 3 q$.
\end{lem}

\begin{proof}

We follow a standard localization argument (e.g., see \cite[Proposition 2.10]{Lee1}). 
For a constant $r$ satisfying $\lambda^{-1}\ll r \ll \delta$, 
let $N_{r\lambda } (A_\lambda(J_i))$ denote the $r\lambda $-neighborhood of $A_\lambda(J_i)$,  $i=1,2,3$. 
Let $\varphi_i$ be a smooth function supported in $N_{r\lambda } (A_\lambda(J_i))$ such that
$\varphi_i = 1$ on $A_\lambda(J_i)$ and $|\partial_\xi^m \varphi_i(\xi)|\lesssim |\xi|^{-|m|}$.

Let us denote by $\mathcal F_x$ the Fourier transform in $x$ and define $\mathcal F_x( \widetilde K_{\lambda,i}(\cdot,t) ) =\mathcal F_x (K_\lambda(\cdot,t) ) \varphi_i $. 
Then $T_\lambda g_i (x,t) = \widetilde K_{\lambda,i}(\cdot, t) \ast g_i(x) $.
We consider a collection $\{ B \} $ of finitely overlapping unit balls  which  {cover} $\mathbb R^2$. 
Fixing $\varepsilon >0$, let  $\widetilde B$ be a ball of radius $\lambda^\varepsilon$ with the same center as $B$.
For each $i$, we decompose $g_i = g_{i,B} + g_{i,B^{\mathrm c}}$ where $g_{i,B} := g_i \chi_{\widetilde B}$ and $g_{i,B^{\mathrm c}} = g_i \chi_{\widetilde{ B}^{\mathrm c} }$. We get  
\begin{align*}
\Big\| \prod_{i=1}^3  T_\lambda    g_i  \Big\|_{L^{\frac q 3}(\mathbb R^2\times I  )}^{\frac q 3} 
 \lesssim \sum_B  \sum_{  \substack{(m_1,m_2,m_3):\\ m_j \in \{B, B^{\mathrm c}\}  }}
\int_I  \int_B  \prod_{i=1}^3  \big| \widetilde    K_{\lambda, i}(\cdot,t) \ast (   g_{i, m_i}     )(x) \big|^{\frac q 3}  dx dt    .
\end{align*}

We consider the case $m_1=m_2=m_3=B$ first. Since the Fourier transform of  $\widetilde    K_{\lambda,i}(\cdot,t) \ast (g_{i,B}  ) $ is supported in $N_{r\lambda } (A_\lambda(J_i) )$,
the transversality condition \eqref{transv} holds for $2^{-1}\delta$ replacing $\delta$ if we take a sufficiently small $r$. 
Applying Proposition \ref{localtri}, we get
\begin{equation*} 
\begin{aligned}
 \sum_B   \int_{B\times I}  \prod_{i=1}^3  \big| \widetilde    K_{\lambda,i}(\cdot,t) \ast (   g_{i,B}     )(x) \big|^{\frac q 3}    dx dt      
\lesssim \lambda^{q \gamma_0} \sum_B 
\prod_{i=1}^3 \| g_i\chi_{\widetilde B} \|_{p}^{\frac q3} \lesssim \lambda^{q \gamma_0 + \tilde c \varepsilon } 
\prod_{i=1}^3  \| g_i \|_{p}^{\frac q 3}
\end{aligned}
\end{equation*}
for $\gamma_0  > \frac12+\frac1p -\frac3q$ and $\tilde c>0$. The second inequality  follows by 
H\"older's inequality and the inclusion $\ell^p \subset \ell^q$ for $p \le q$ since the balls $\widetilde B$ overlap  at most  $C\lambda^{2\varepsilon}$.

If $m_i=B^{\mathrm c}$ for some $i$,  we use decay of $\widetilde K_{\lambda,i}$. 
As before, it is easy to show 
$|\widetilde K_{\lambda, i}(x,t)| \lesssim \lambda^2(1 + |x| )^{-M}$ for any $M\ge1$ (c.f., \eqref{kernelB}).   
Setting $\mathscr K (x ) = (1 +|x|)^{-3}$, we see $|\widetilde K_{\lambda,i} (\cdot,t)\ast g_{i,B} (x) |   
\lesssim \lambda^2    \mathscr K\ast |  g_i|(x) $
and
$\big| \widetilde K_{\lambda,i} (\cdot,t)\ast g_{i,B^{\mathrm c}} (x) \big|  
\lesssim \lambda^2  \lambda^{ \varepsilon (3-M)}  \mathscr K\ast | g_i|(x) $  if  $x \in B$. 
Thus we have 
\[ \prod_{i=1}^3  \big| \widetilde    K_{\lambda,i}(\cdot,t) \ast (   g_{i, m_i}  )(x) \big| \lesssim  \lambda^{c_1 -c_2 \varepsilon M}   \prod_{i=1}^3  \mathscr K\ast |  g_i|(x), \quad  (x,t) \in B\times I, \] 
for some constants $c_1, c_2>0$. 
By H\"older's inequality and Young's convolution inequality, there are positive constants $c_1, c_2$ such that
\begin{align*}
\sum_B  \sum_{ (m_1,m_2,m_3)\neq (B, B, B)} \iint_{B \times I}  \prod_{i=1}^3  \big| \widetilde    K_{\lambda,i}(\cdot,t) \ast (   g_{i, m_i}     )(x) \big|^{\frac q 3}   dx dt  
\lesssim
\lambda^{c_1 -c_2 \varepsilon M}  \prod_{i=1}^3  
\|     g_i \|_{p}^{\frac q3} 
\end{align*}
for $p\le q$. Combining the estimates above,
we obtain 
\[
\Big\| \prod_{i=1}^3  T_\lambda g_i  \Big\|_{L^{\frac q3}(\mathbb R^2\times I )}^{\frac q3} 
\lesssim \big( \lambda^{   q  \gamma_0
+\tilde c \varepsilon}  +  \lambda^{c_1 -c_2\varepsilon M}  \big)
	\prod_{i=1}^3  \| g_i \|_{p}^{\frac q 3}
\]
for $\gamma_0  > \frac12+\frac1p -\frac3q$. Taking $\varepsilon>0$ small enough and then a large $M$, 
we obtain the desired estimate \eqref{pqest}.
\end{proof}

We now obtain \eqref{pqest}  with weights $\omega\in \Omega(\alpha)$ via the following lemma. 

\begin{lem}\label{elementary}
Let $\omega \in \Omega(\alpha)$, $0<\alpha \le 3$.
Suppose $\widehat F$ is supported in $\mathbb B^3(0,\lambda)$. Then, we have 
$\| F\|_{L^q(\mathbb R^3; \omega)} \lesssim \lambda^{(3-\alpha)/q}\|F\|_{L^q(\mathbb R^3)}$ for $q\ge1$.
\end{lem}
\begin{proof}
Let $\varphi$ be a Schwartz function such that $\widehat \varphi=1$ on $\mathbb B^3(0,1)$,
and set $\varphi_\lambda = \lambda^3 \varphi(\lambda \cdot)$.
Since $F=F \ast \varphi_\lambda$, we have $\|F\|_{L^q(\mathbb R^3; \omega)} \le \|F\|_{L^q(\mathbb R^3 )} \| |\varphi_\lambda| \ast \omega\|_{\infty}^{1/q}$.
By rapid decay of $\varphi $, we see
\begin{align*}
|\varphi_\lambda | \ast \omega (y) 
 \le C_N  \lambda^3 \sum_{\ell\ge 0} 2^{-N\ell} \int_ {\mathbb B^3(y,2^\ell \lambda^{-1} )}   \omega (z) dz  
\end{align*} 
for any   $N>3$.  Since $\omega \in \Omega(\alpha)$,  it follows by \eqref{omega}  that  $\| |\varphi_\lambda| \ast \omega\|_{\infty} \lesssim \lambda^{3-\alpha}$. Hence, we get the desired bound.
\end{proof}

\begin{prop}\label{w-globtri} Let $\phi \in \mathfrak S(\epsilon_\circ)$, $\omega \in \Omega(\alpha)$, $0<\alpha\le 3$, and  $\lambda^{-1} \ll \delta \ll 1$.  
Suppose \eqref{transv} holds whenever $\theta_i \in J_i$, $i=1,2,3$.  
If \eqref{pqest} holds for some $\gamma_0$, then  
\begin{equation}\label{w-tri}
\Big\| \prod_{i=1}^3    T_\lambda g_i \Big\|_{L^{\frac q 3} (\mathbb R^2\times I;\omega)}
\le C   \lambda^{3 \gamma}   
\prod_{i=1}^3 \| g_i \|_p    
\end{equation}
holds for $\gamma= \gamma_0+\frac{3-\alpha}q $ whenever $\supp \widehat{g_i}\subset A_\lambda(J_i)$.
\end{prop}

\begin{proof}
Let $\widetilde \chi \in \mathrm C_0^\infty((1/2,4))$ such that $\widetilde \chi =1$ on $I=[1,2]$.
We only have to  show \eqref{w-tri} with $T_\lambda$ replaced by $\widetilde{\chi}\, T_\lambda$.
Since $\widetilde\chi$ is compactly supported,  the support of the space-time Fourier transform of $   \widetilde \chi(t) T_\lambda g_i(x,t)$ is unbounded.   
So, in order to apply Lemma \ref{elementary}, 
we decompose $\widetilde \chi \, T_\lambda f$ in such a way that the Fourier supports of the consequent operators are contained in either a bounded set or  its complement.
Let us define a frequency localized operator  $T_\lambda^1$ by 
\begin{align*}
 \mathcal F_{x,t} \big( T_\lambda^1 g \big) (\eta,\rho,\tau)   = \mathcal F_{x,t} \big(  \widetilde \chi   \, T_\lambda g \big)(\eta,\rho,\tau) \beta_0 (  (c\lambda)^{-1}|\tau| ),
\end{align*}
where $c>0$ is a constant  to be chosen later. 
We also set $T_\lambda^2 g =   (  \widetilde \chi   \, T_\lambda - T_\lambda^1) g $.  Then, we have  
\begin{align*}
& \Big\|  \prod_{i=1}^3   \widetilde \chi  \,   T_\lambda g_i \Big\|_{L^{\frac q 3} (\mathbb R^3;\omega)}  \le {\rm I} + {\rm I\!I}, 
\end{align*}
where 
\begin{align*} 
{\rm I} = \Big\|  \prod_{i=1}^3 T_\lambda^1 g_i \Big\|_{L^{\frac q 3} (\mathbb R^3;\omega)},  \quad
{\rm I\!I} = \sum_{\substack{(m_1,m_2,m_3)\neq(1,1,1)  :\\ m_j \in\{1,2\}    } } \Big\| \prod_{i=1}^3   T_\lambda^{m_i} g_i \Big\|_{L^{\frac q3} (\mathbb R^3;\omega)} . 
\end{align*}

First, we show ${\rm I\!I} \lesssim \lambda^{-M} \prod_{i=1}^3 \| g_i \|_{p}$ for any $M >1$. It is enough to show
\begin{align}\label{T2}
\|  T_\lambda^{2}g \|_{ L^{q} (\mathbb R^3;\omega)} \lesssim \lambda^{-M} \|g\|_p
\end{align}
for any $M>0$ if $p\le q$. 
Once we have \eqref{T2}, it follows that  $\| T_\lambda^1g \|_{L^q(\mathbb R^3;\omega)} \lesssim  \lambda^{2}\|g\|_p$ 
since $|T_\lambda^1g | \le | \widetilde \chi   \, T_\lambda g| + |T_\lambda^2g|$ and $\|  \widetilde \chi  \,  T_\lambda g \|_{L^{q} (\mathbb R^3;\omega)} \lesssim \lambda^2 \|g\|_{p}$ for $p\le q$ (see \eqref{kernelB} and \eqref{claim01}).
Thus, the desired bound on $\mathrm{I\!I}$ follows by H\"older's inequality. 

To prove \eqref{T2}, let us set
\[
\Phi (\eta,\rho,\tau) = \frac{1}{2\pi}\int e^{-i t ( \tau -  \eta \phi(\rho/\eta))} \widetilde \chi(t) dt \,   \beta(\lambda^{-1}\eta)\beta_0(|\rho/\eta|).
\]  
Recalling the definition of $T_\lambda$ (\eqref{TR}), we have $\mathcal F_{x,t}( \widetilde \chi  \, T_\lambda g) (\eta,\rho,\tau)= \Phi(\eta,\rho,\tau) \widehat g(\eta,\rho)$. 
Note $T_\lambda^2 = ( \widetilde \chi  \,  T_\lambda - T_\lambda^1)$, so
$
  T_\lambda^2 g(x,t) = \mathfrak K(\cdot, t)\ast g(x),
$
where 
\[\mathfrak K(x ,t) = \iiint \big(1 -\beta_0((c\lambda)^{-1} |\tau|)\big) \Phi(\eta,\rho,\tau)   e^{i (x \cdot(\eta,\rho) + t\tau)} d\eta d\rho d\tau . \]   
Taking a sufficiently large $c>0$, by integration by parts, we get
\begin{equation}\label{k-decay}
|\mathfrak K(x,t) | \lesssim \lambda^{-M} (1 + |x|)^{-M} (1 + |t|)^{-M}
\end{equation}
for $M\ge1$. 
We see $\partial_{\eta,\rho,\tau}^m \big( (1 -\beta_0((c\lambda)^{-1} |\tau|))\Phi(\eta,\rho,\tau)  \big)$ are bounded by $C_M \lambda^{-M} (1+ |\tau|)^{-M}$ for any $M \ge 1$ since $\partial_{\eta,\rho}^m\phi$ are bounded by a constant depending only on the multi index $m$. Repeated integration by parts in $\eta,\rho,\tau$ gives \eqref{k-decay}. 
Using \eqref{k-decay}, we get \eqref{T2} with $p=q$ by Schur's test. Also, \eqref{T2} holds with $q=\infty$ by Young's convolution inequality. By interpolation, we get \eqref{T2} for any $1 \le p \le q \le \infty$.

We now turn to  the term ${\rm I}$. Since the Fourier transform of $T_\lambda^1 g_i $ is supported in a ball of radius $2c \lambda$, we can apply Lemma \ref{elementary} to $\rm I$. So, we get
\begin{align*} 
  {\rm I}  
 \lesssim \lambda^{\frac{3(3-\alpha)}{q}} \Big\|  \prod_{i=1}^3 T_\lambda^1 g_i \Big\|_{L^{\frac q 3}  (\mathbb R^3 )}. 
\end{align*}
Since $|T_\lambda^1g|\le | \widetilde \chi  T_\lambda g| + |T_\lambda^2 g|$, using $\|  \widetilde \chi  \, T_\lambda g \|_{L^q(\mathbb R^3)} \lesssim \lambda^2\|g\|_p $ and \eqref{T2}, we see that $ \big\| \prod_{i=1}^3 T_\lambda^1 g_i \big\|_{ q/ 3}$ is bounded by $ \|  \prod_{i=1}^3  \widetilde \chi  \,  T_\lambda g_i \|_{q/3}+ C \lambda^{-N} \prod_{i=1}^3 \|g_i\|_{p}$ for any $N>0$. Now we apply  Lemma \ref{globtri} to $ \|  \prod_{i=1}^3  \widetilde \chi  \,  T_\lambda g_i \|_{q/3}$ expanding the interval $I$ slightly. 
Choosing $N$ large enough, we obtain 
\begin{align*}
 {\rm I}  \lesssim \lambda^{3(\gamma_0  +  \frac{3-\alpha} q) }  \prod_{i=1}^3 \|g_i\|_{p}. 
\end{align*}
Combining all the bounds on ${\rm I}$ and ${\rm I\!I}$, we get the desired estimate \eqref{w-tri} by taking  $M$ sufficiently large.
\end{proof}
  
The estimate \eqref{w-tri} plays a crucial role in proving Theorem \ref{LS-est}. However, in order to prove Theorem \ref{smoothing} (or Proposition \ref{trg}), it suffices to obtain the following estimate which is  not necessarily sharp.

\begin{cor}\label{globtri22}
Let $\phi \in \mathfrak S(\epsilon_\circ)$, $\omega \in \Omega(\alpha)$ for $1<\alpha \le 2$, and $\lambda^{-1} \ll \delta \ll 1$. 
Suppose \eqref{transv} holds whenever $\theta_i \in J_i$, $i=1,2,3$, and let \begin{equation}\label{tri-pq} 
 \frac1q \le \frac1p <\frac\alpha q ~ \text{ and }~
  p > \begin{cases}
 \,  \frac{2\alpha+3}{2\alpha}  , ~& \quad \hfill \frac32 <\alpha \le 2, \\
 \,  \frac{6-2\alpha}{\alpha} , ~&\quad   \hfill 1<\alpha \le \frac32.
 \end{cases}
 \end{equation}
Then we have \eqref{w-tri} for some $\gamma<\frac12$ 
whenever $\supp \widehat{g_i}\subset A_\lambda(J_i)$.
\end{cor}

\begin{proof}
By Proposition \ref{localtri} and Proposition \ref{w-globtri}, we  have 
\begin{equation}\label{gg}
\Big\| \prod_{i=1}^3  T_\lambda g_i \Big\|_{L^{\frac q 3} (\mathbb B^2(0,1)\times I;\omega)}
\le C  \lambda^{ 3\gamma}
\prod_{i=1}^3 \| g_i \|_p   
\end{equation}
for $\gamma'> \frac12 +\frac1p -\frac \alpha q$ provided that $\frac 1q \le \min(\frac 1p,~ \frac{1}{3p}+\frac 16,~ \frac 2{3p'})$. 
By H\"older's inequality,  \eqref{gg} continues to be true for some $\gamma'<\frac12$ if $p,q$ satisfy \eqref{tri-pq}.  
Indeed, there is a $q$  such that   $\frac1p =\frac\alpha q$ for $p> \frac{6-2\alpha}{\alpha}$  when $1<\alpha \le \frac32 $, and  for $p>\frac{2\alpha+3}{2\alpha}$ when  $ \frac32 <\alpha \le 2$. 

Repeating the same argument as in the proof of Lemma \ref{globtri},  one can extend \eqref{gg} to the global estimate \eqref{w-tri} as long as $q \ge p$. We omit the details. 
\end{proof}

\subsection{Proof of Proposition \ref{trg}}
In this subsection, we prove Proposition \ref{trg} by combining Lemma \ref{Lscale} and Corollary \ref{globtri22}.
To this end, we decompose the frequency support of $T_\lambda^\phi g$ independent of particular choice of $\phi$. 

\subsubsection*{Decomposition}
Let $K_i$, $i =0,1,2$ be dyadic numbers such that
\[ \lambda^{-1/3} \ll   K_2 \ll K_1 \ll K_0 =1. \]
For each $K_i$, we consider a collection $\mathfrak J^i$ of dyadic  intervals $J^i=[(j-1)K_i, (j+1)K_i]$ for $j \in \mathbb Z$, $|j| \le K_i^{-1}$ such that union of $J^i$ covers $[-1,1]$.

Let $\varrho \in \mathrm C_0^\infty((-1,1))$ satisfying $\sum_{j \in \mathbb Z} \varrho(\cdot-j)=1$.
For an interval $J=[\theta-K,\theta+K]$, we denote $\varrho_{J}^{}=\varrho(K^{-1}(\cdot-\theta))$.
Note that $\sum_{J^i \in \mathfrak J^i} \varrho_{J^i}^{} =1$ on $[-1,1]$.
We set 
$\widehat {g^{}_{J^i}}(\eta,\rho)=\widehat g (\eta,\rho) \, \varrho_{J^i}^{}(\rho/\eta)$.
Then, we have 
\[
T_\lambda^\phi g(x,t) = \sum_{J^i \in \mathfrak J^i} T_\lambda^\phi g^{}_{J^i}(x,t), \quad i=0,1,2.
\]
Note that $\widehat {g^{}_{\!J^i}}$ is supported in a rectangle of dimensions $c \lambda \times c \lambda K_i$ for a constant $c>0$.  

Following the argument in \cite[Section 3]{LV}, we have the next lemma. 
\begin{lem}
For each $(x,t) \in \mathbb R^2\times I$, there exists a constant $C>0$, independent of $K_i$ and $(x,t)$, such that  
\begin{equation}\label{divideU}
\begin{aligned}
|T_\lambda^\phi g(x,t)|
\le   \sum_{i=1}^2 C K_{i-1}^{-2} &\max_{J^i\in \mathfrak J^i} |T_\lambda^\phi  g^{}_{J^i}(x,t)| \\
+&C K_2^{-4} \max_{ (J^2_1,J^2_2,J^2_3) \in \mathfrak J^2(K_2) } \prod_{i=1}^3|T_\lambda^\phi  g^{}_{J^2_i} (x,t) |^{\frac 13}  ,
\end{aligned}
\end{equation}
where $  \mathfrak J^2( K_2) :=  \{ (J^2_1, J^2_2, J^2_3 ): J^2_1,J^2_2,J^2_3 \in \mathfrak J^2,~ 
 \min_{k \neq l} \dist(J^2_k,J^2_l) \ge K_2  \} $.
\end{lem}

If $(J^2_1,J^2_2,J^2_3) \in \mathfrak J^2(K_2)$, then 
 \eqref{transv} holds for $\theta_i \in J_i^2$, $i=1,2,3$. More precisely, we have the following. 
\begin{lem}\label{lem:transv}
Let $J_1,J_2,J_3$ be subintervals of $[-1,1]$ such that $\min_{k \neq l} \dist (J_k,J_l)  \gtrsim \delta^{\frac13} $. Suppose that $\phi \in \mathfrak S(\epsilon_\circ)$. 
If $\epsilon_\circ\in (0,1/8)$, then 
$ |\det \big( \mathrm N (\theta_1), \mathrm N(\theta_2), \mathrm N(\theta_3)\big)| \gtrsim \delta$ for $\theta_i \in J_i$, i=1,2,3.
\end{lem}
\begin{proof} 
By the generalized mean value theorem (see \cite[Part V, Problem 95]{PolyaSzego}),
there are $u_j$ for $j=1,2,3$ such that $\min_{1\le i \le 3} \theta_i < u_j <\max_{1\le i \le3} \theta_i$
and
\[
\det (\mathrm N(\theta_1),\mathrm N(\theta_2),\mathrm N(\theta_3)) = \det (\mathrm N(u_1), \mathrm N'(u_2), \mathrm N''(u_3)) \prod_{1 \le k<\ell \le3}
|\theta_k-\theta_\ell|.
\]
From \eqref{normal} we have 
\begin{align*}
\det (\mathrm N(u_1), \mathrm N'(u_2), \mathrm N''(u_3))=  \det
\begin{pmatrix}
 u_2 \phi''(u_2) & \phi''(u_3) + u_3 \phi^{(3)}(u_3)\\
\phi''(u_2) & \phi^{(3)}(u_3)
\end{pmatrix}.
\end{align*}
Since $\phi \in \mathfrak S(\epsilon_\circ)$ and $\epsilon_\circ \in (0,1/8)$, $7/8\le |\phi''(\theta)| \le 9/8 $ and $|\phi^{(3)}(\theta)| \le 1/8$ for $\theta \in [-1,1]$. Thus, it  follows that $|\det(\mathrm N(u_1),\mathrm N'(u_2),\mathrm N''(u_3))|\ge 27/8^2$. 
So we obtain the desired bound since $\prod_{1 \le k <\ell \le 3}|\theta_k-\theta_\ell| \gtrsim \delta$.
\end{proof}

We are now  ready to prove Proposition \ref{trg}.

\begin{proof}[Proof of Proposition \ref{trg}] 
We begin by making a primary decomposition by which we reduce the matter
 to obtaining estimate for $T_\lambda^{\phi}$, $\phi \in \mathfrak S(\epsilon_\circ)$ relative to $\omega \in \Omega(\alpha).$ 

Let $\epsilon_\circ\in (0, 1/8)$. Following the argument in the proof of Lemma \ref{para}, 
 we  can fix a small positive dyadic number  $h_\ast<h_\circ(\epsilon_\circ)$ such that 
$(\psi_{\mathsmaller{0}})_{\theta, h_\ast}\in \mathfrak S(\epsilon_\circ)$ for $\theta \in [-1,1]$.
Let $J_j=[(j-1)h_\ast, (j+1)h_\ast]$, $j \in \mathbb Z$ such that $|j| \le h_*^{-1}+2$.
Then $\sum_{j} \varrho_{J_j}^{}=1$ on $[-1,1]$.
Denoting $
\widehat{ g_j}(\eta,\rho)  = \widehat g(\eta,\rho) \varrho^{}_{ J_j}(\rho/\eta)$, 
we get
\[
\|T_\lambda^{\psi_{\mathsmaller{0}}} g\|_{L^q(\mathbb R^2 \times I ;\omega)} \le 
\sum_j \| T_\lambda^{\psi_{\mathsmaller{0}}} g_j \|_{L^q(\mathbb R^2 \times I ;\omega)} .
\]

Let $\theta_j= jh_*$. 
As before (cf. \eqref{T-rescale}), the change of variable $(\eta,\rho ) \rightarrow L_{\theta_j,{h_\ast}}(\eta,\rho)$ gives 
\begin{align}
 \|T_\lambda^{\psi_{\mathsmaller{0}}} g_j \|_{L^q(\mathbb R^2\times I;\omega)}^q  
&  \le C h_\ast^{2\alpha - \kappa(\alpha)  } 
\big\|  T_{\psi_{\mathsmaller 0}''(\theta_j) h^2_\ast\lambda}^{(\psi_{\mathsmaller{0}})_{\theta_j,{h_\ast}}} (g_j)_{\theta_j,{h_\ast}} \big\|_{L^q(\mathbb R^2\times I ;\omega_{\theta_j,{h_\ast}})}^q.
\end{align}
Note $\| (g_j)_{\theta_j,{h_\ast}}\|_p=h_*^{-3/p}\|g\|_p$ and $\omega_{\theta_j,{h_\ast}} \in \Omega(\alpha)$.
Since $|\psi_{\mathsmaller 0}''(\theta_\circ)-1|\le \epsilon_\circ$ and $h_\ast$ is a fixed constant depending only on $\epsilon_\circ$, 
in order to prove Proposition \ref{trg} it suffices to show 
\begin{equation}\label{TR-est2}
\| T_\lambda^{\phi} g \|_{L^q(\mathbb R^2\times I; \omega )} \le C \lambda^{\gamma} \|g\|_{p}
\end{equation}
for some $\gamma<1/2$ provided that $\phi \in \mathfrak S(\epsilon_\circ)$ and $\omega \in \Omega(\alpha).$ 

Since $1<\alpha\le 2$, we have $\frac{\alpha-1}{q} - \frac3 p +1 >0$ from the hypothesis. So we can choose $\gamma <\frac12$ such that 
\Be
\label{gamma}
 \frac{ \alpha-1 }{q} -\frac3p  +  2\gamma >0. 
\Ee
For such $\gamma$ and for $\lambda'\ge 1$ we set
\begin{equation*} 
 \mathfrak Q (\lambda')  =\mathfrak Q  (\lambda',p,q,\alpha) := \sup_{1\le\lambda \le \lambda'} {\lambda}^{- \gamma  }  Q({\lambda}), 
 \end{equation*}
 where  $Q({\lambda})$ is defined  by \eqref{QR}. The estimate \eqref{TR-est2} follows if we show  $\mathfrak Q (\lambda') \le C $.

Let  $h_\circ$  be a small positive number so that  Lemma \ref{para} and Lemma \ref{Lscale} hold. 
Let $K_1,$ $K_2$ be positive dyadic numbers such that $0< K_2 \ll K_1\le h_\circ $. We also set $K_0=1$ as before.  
Using the decomposition \eqref{divideU}, we have
\begin{equation}\label{eq2} 
\begin{aligned}
	\|T_\lambda^\phi g \|_{L^q(\mathbb R^2\times I;\omega)} 
\le   \mathrm S_1+\mathrm S_2+\mathrm S_3, 
\end{aligned}
 \end{equation} 
 where 
 \begin{align*}
 \mathrm S_i&=  C K_{i-1}^{-2} \big\| \max_{J^i\in \mathfrak J^i} |T_\lambda^\phi  g^{}_{J^i} | \big\|_{L^q(\mathbb R^2\times I;\omega)}, \quad i=1,2
 \\
 \mathrm S_3&= C K_2^{-4} \Big\| \max_{(J^2_1,J^2_2,J^2_3) \in \mathfrak J^2(K_2)
 } \prod_{i=1}^3|T_\lambda^\phi  g^{}_{J^2_i}   |^{\frac 13}   \Big\|_{L^q(\mathbb R^2\times I;\omega)} .
 \end{align*}

Note that 
$\mathrm S_i^q  \le  C^q K_{i-1}^{-2q}
\sum_{J^i \in \mathfrak J^i} \| T_\lambda^\phi g^{}_{J^i}\|_{L^q(\mathbb R^2\times I;\omega)}^q$. 
By Lemma \ref{Lscale} with $\kappa(\alpha) = \alpha+1$, 
for $i=1,2$ 
we have 
\begin{align*}
\mathrm S_i^q 
&   \le C' K_{i-1}^{-2q} \big( K_i^{\frac{\alpha -  1 }q - \frac 3p  }
\sup_{\theta \in [-1,1]}Q( \phi''(\theta)K_i^2 \lambda) \big)^q  \sum_{J^i \in \mathfrak J^i}\| g_{J^i}^{} \|_p^q.
\end{align*} 
By the embedding $\ell^p \subset \ell^q$ ($p\le q$), it follows that $ \sum_{J^i \in \mathfrak J^i}\|g^{}_{J^i}\|_p^q  \le \| g \|_p^q$.
Here, we use the estimate $
\big( \sum_{J^i \in \mathfrak J^i} \| g_{J^i}^{} \|_p^p\big)^{1/p} \lesssim \|g\|_p
$ for $2 \le p \le \infty$, which can be shown by interpolation between trivial $L^\infty$--$\ell^\infty L^\infty$ estimate and $L^2$--$\ell^2L^2$ estimate via Plancherel's theorem.

By Lemma \ref{lem:transv},    the transversality condition \eqref{transv} with $\delta=(K_2)^3$ holds if $(J^2_1,J^2_2,J^2_3) \in \mathfrak J^2(K_2)$.  
By Corollary \ref{globtri22},  we have $\mathrm S_3  \le  CK_2^{-C}  \lambda^{\gamma }   \| g  \|_p$ for
$\lambda \gg K_2^{-3}$ and $p,q$ satisfying \eqref{tri-pq}.  
For the case of  $\lambda \lesssim  K_2^{-3}$, we have $\mathrm S_3  \le C K_2^{-C} \lambda^\gamma   \| g  \|_p$ for some constant $C>0$ from the easy estimate (e.g. \eqref{claim01}) and H\"older's inequality.  
So we get
\[ \mathrm S_3  \le  C K_2^{-C}  \lambda^{\gamma }  \| g  \|_p . \]

Combining all the estimates for $\mathrm S_1$, $\mathrm S_2$, and $\mathrm S_3$, which hold uniformly  for $\phi \in \mathfrak S(\epsilon_\circ)$ and $\omega \in \Omega(\alpha)$, we  obtain
\begin{align} 
\label{qqq}
Q(\lambda)
&\le 
 \sum_{i=1}^2 C K_{i-1}^{-2} K_i^{\frac{\alpha -  1 }q - \frac 3p  }
\sup_{\theta \in [-1,1]} Q( \phi''(\theta)K_i^2 \lambda)  
 +C   K_2^{-C}   \lambda^{\gamma}  .
\end{align}
Since $|\phi''(\theta)-1|\le\epsilon_\circ$, note that $  (\phi''(\theta)K_i^2 \lambda)^{-\gamma}   Q( \phi''(\theta)K_i^2 \lambda)   \le \mathfrak Q(\lambda')$ if $\lambda' \ge 2^{-1}K_i^2 \lambda\ge 1$.
Otherwise, i.e., if $1\le  \lambda\le 2K_i^{-2}$, we have $Q( \phi''(\theta)K_i^2 \lambda)\le  CK_i^{-C}$ for some $C>0$.  
Thus, we have $  (\phi''(\theta)K_i^2 \lambda)^{-\gamma}   Q( \phi''(\theta)K_i^2 \lambda)  
\le \mathfrak Q(\lambda')+  CK_i^{-C}  $.

We now multiply $\lambda^{-\gamma} $ to both sides of \eqref{qqq}.  Then, using  the above observation we  obtain 
\begin{align*}
\lambda^{-  \gamma } Q(\lambda)
  \le   \sum_{i=1}^2 C K_{i-1}^{-2} K_i^{\frac{  \alpha -1 }q -\frac 3p   + 2\gamma}  
\mathfrak Q (\lambda')
+    C  K_2^{-C} 
\end{align*}
for $\lambda\le \lambda'$. 
Taking supremum over $\lambda \le \lambda'$, we get
\[\mathfrak Q  (\lambda')     \le  \sum_{i=1}^2 CK_{i-1}^{-2} K_i^{\frac{  \alpha -1 }q -\frac 3p +   2\gamma } 
\mathfrak Q (\lambda')
+    C  K_2^{-C}.  \]
Since \eqref{gamma} holds,   we  can successively choose dyadic numbers $K_1, K_2$  such that 
\[ CK_{i-1}^{-2} K_i^{\frac{  \alpha -1 }q -\frac 3p  + 2\gamma } \le \,3^{-1}, \quad i=1,2.\]
Therefore, we obtain $\mathfrak Q  (\lambda')  \le \frac 23 \mathfrak Q  (\lambda')+ C  K_2^{-C}$, which clearly gives  $\mathfrak Q  (\lambda')  \le  C  K_2^{-C}$  for $\lambda'\ge 1$
and $p,q$ satisfying \eqref{tri-pq} and $\frac{\alpha-1}{q} - \frac 3p + 1>0$, i.e., \eqref{pq}.
This completes the proof. 
\end{proof}

\subsection{Weighted local smoothing estimate with sharp regularity}
In this section, modifying the proof of Proposition \ref{trg}, we show  \eqref{TR-est} with $\gamma > \frac12 +\frac1p -\frac{\alpha}{q}$, which implies Theorem \ref{LS-est}. 
We also discuss sharpness of the exponent $\gamma$. 

\begin{prop}\label{trg22}
Let  $0 < \alpha \le 3$  and $\omega \in \Omega(\alpha)$. 
If  
$\frac1q \le \min\big( \frac1p , ~\frac1{3p}+\frac16,~ \frac2{3p'} \big)$ and $ \frac1p +    \frac{\kappa(\alpha)}{q}  \le 1 $,  
then \eqref{TR-est} holds for $\gamma>\frac12+\frac1p -\frac\alpha q$.
\end{prop} 

\begin{proof} 
 The proof is essentially identical to that of Proposition \ref{trg}, so we shall be brief. 
 The main difference is that we use Proposition \ref{localtri} and  Proposition \ref{w-globtri} instead of Corollary \ref{globtri22}.
For a fixed  $\gamma > \frac12+\frac1p-\frac\alpha q$, 
we need to show  $  \mathfrak Q  (\lambda') \le C $.

As in the proof of Proposition \ref{trg}, we use the decomposition \eqref{divideU}. 
Since $\sum_{J^i}\|g_{J^i}^{}\|_p^q \lesssim \|g\|_p^q$, by Lemma \ref{Lscale}  we have  
\begin{align*}
\big\| \max_{J^i\in \mathfrak J^i}   |T_\lambda  g^{}_{J^i} | \big\|_{L^q(\mathbb R^2\times I;\omega)} 
  \le C K_i^{\frac{2\alpha - \kappa(\alpha) }q - \frac 3p  }
Q( K_i^2 \lambda)  \| g \|_{p} .
\end{align*}
By Proposition \ref{localtri} and Proposition \ref{w-globtri}  we also have
\begin{align*}
 \Big\| \max_{ (J^2_1,J^2_2,J^2_3) \in \mathfrak J^2(K_2)} \prod_{i=1}^3 \big|T_\lambda  g^{}_{J^2_i} \big|^{\frac 13}   \Big\|_{L^q(\mathbb R^2\times I;\omega)} 
 \le CK_2^{-C}  \lambda^{\gamma}   \| g  \|_{p} 
\end{align*}
for $\gamma >\frac12 +\frac1p -\frac{\alpha}{q}$ if $\frac 1q \le \min(\frac 1p, \frac{1}{3p}+\frac 16, \frac 2{3p'})$ and $1\le p\le \infty$.

By following the same lines of argument  as in the proof of Proposition \ref{trg}
we obtain 
\[
 \mathfrak Q  (\lambda')     \le C  \sum_{i=1}^2 K_{i-1}^{-2} K_i^{\frac{2\alpha - \kappa(\alpha) }q - \frac 3p   + 2\gamma  } 
 \mathfrak Q  (\lambda')
+    C K_2^{-C}
\]
for $\lambda'\ge 1$. 
We take $\gamma$ sufficiently close to $\frac12+\frac1p -\frac\alpha q$ such that $ \frac{2\alpha - \kappa(\alpha) }q - \frac 3p   + 2\gamma   > 1- \frac 1p - \frac{  \kappa(\alpha) }q \ge 0$.
Then, we can choose $K_1, K_2$ such that $CK_{i-1}^{-2} K_i^{\frac{2\alpha - \kappa(\alpha) }q - \frac 3p   + 2\gamma  } \le \frac13$ for $i=1,2$. 
Therefore, 
 $ \mathfrak Q  (\lambda') \le C$ holds provided that
$\frac1q \le \min(\frac1p , \frac1{3p}+\frac16,\frac{2}{3p'})$ and $1 - \frac1p  -   \frac{\kappa(\alpha)}{q}  \ge 0 $, as desired.
\end{proof}

We discuss sharpness of the regularity exponent in Theorem \ref{LS-est}.  
We show \eqref{LS-AAA} holds only if 
\begin{align} 
\label{gamma-con}\gamma & \ge  \frac 12+ \frac 1 p -\frac\alpha q,
\\
\label{gamma-con2}
\gamma  &\ge \frac{3}{2p} + \frac{\kappa(\alpha)-2\alpha}{2q}. 
\end{align}
Note the  two lower bounds on $\gamma $ coincide if $\frac1 p + \frac{\kappa(\alpha)}{q} =1$.
So,  \eqref{LS-AAA} 
	holds for $\gamma >\frac12 +\frac1p -\frac{\alpha}{q}$ only if 
$\frac1p +    \frac{\kappa(\alpha)}{q}  \le 1 	$.

\begin{proof}[Proof of \eqref{gamma-con}]
For $\lambda \gg1$, we consider $f$ given by $\widehat f(\xi)=e^{-i|\xi|} \zeta(\lambda^{-1}\xi)$ for a Schwartz function $\zeta$ supported on $1\le|\xi| \le 2$.
Using the polar coordinates  (see \eqref{ker-wave}), we have
\[
|f(x)| \lesssim \lambda^{\frac 3 2}  (1+\lambda|	|x|-1|)^{-N}
\]
for any $N\ge1$. Since $\supp \widehat f$ is  supported in $\{ |\xi| \sim \lambda \}$,  
we see
$
\|f\|_{L_\gamma^p} 
\lesssim \lambda^\gamma \lambda^{\frac32-\frac1p}.
$
Note that $|e^{it\sqrt{-\Delta}} f(x) |\gtrsim \lambda^{2}$ if $|x|, |t-1| \le c\lambda^{-1}$ for a sufficiently small $c>0$.
If we take  
\[d\nu(z) =\chi_{\mathbb B^3(0,3)}(z) |z|^{\alpha-3} dz, \quad z=(x,t) \in \mathbb R^3,\] then
it is easy to see $\langle\nu\rangle_\alpha\le C$ for some $C$.  So,
 \eqref{LS-AAA} gives $\lambda^{2 - \frac \alpha q} \lesssim \lambda^\gamma \lambda^{\frac 32-\frac1p}$. Taking $\lambda\to \infty$, we get   \eqref{gamma-con}. 
\end{proof}

\begin{proof}[Proof of \eqref{gamma-con2}]

We use a Knapp type example. 
Let us take $f$ such that $\widehat f(\xi) = \lambda^{-\frac 32} \zeta(\lambda^{-1}\xi_1, \lambda^{-\frac 12}\xi_2)$.
Then $|e^{it\sqrt{-\Delta}} f(x)| \gtrsim 1$  for $(x,t) \in R_\lambda$, where 
\[R_\lambda :=\{ (x,t): |x_1-t| \le c\lambda^{-1},~ |x_2| \le c \lambda^{-\frac 12},~  t  \in I\}\]
for a sufficiently small $c>0$.
We consider
\[ d\nu (x,t)= \lambda^{\frac{3-2\alpha+\kappa(\alpha)}{2}} \chi_{R_\lambda}(x,t)\,dxdt.\] It is easy to show  $\langle \nu \rangle_\alpha \lesssim 1$. 
Indeed, if $0<r \le \lambda^{-1}$, we have $\nu(\mathbb B^3(z,r)) \lesssim \lambda^{\frac{3-2\alpha+\kappa(\alpha)}2}r^3 \le r^\alpha$
since $r^{\frac 32} \le r^{\frac {\kappa(\alpha)}2}$.  One can show  $\nu(\mathbb B^3(z,r))\lesssim r^\alpha$  for the cases $\lambda^{-1} <r \le \lambda^{-\frac 12}$ and $\lambda^{-\frac 12} <r \le1$ similarly.
Thus, the inequality  \eqref{LS-AAA} gives $ \lambda^{\frac{-2\alpha+\kappa(\alpha)}{2q}}
\lesssim \lambda^{ \gamma } \lambda^{-\frac{3}{2p}}$, which implies 
\eqref{gamma-con2}. 
\end{proof}

\section{Proofs of Theorem \ref{thm:circular-avr}, Theorem \ref{smoothing}, and Theorem \ref{LS-est} }\label{Sec:3}
In this section, we first prove Theorem \ref{smoothing} and Theorem \ref{LS-est}, which we deduce from Proposition \ref{trg} and  Proposition \ref{trg22}, respectively.
Combining  Littlewood-Paley decomposition and Theorem \ref{smoothing}, we  prove Theorem \ref{thm:circular-avr}, and then Corollary \ref{thm:circular}.  We also discuss results for the spherical average in higher dimensions.

\subsection{Proof of Theorem \ref{smoothing} and Theorem \ref{LS-est}}
By Littlewood-Paley decomposition, it suffices to consider the operator $e^{it\sqrt{-\Delta}}P_\lambda f$ with $\lambda\ge 1$.
We first observe that \eqref{TRg} is a simple consequence of the estimate \eqref{TR-est} via a change of variable.
Let $U(x,t)= (x_1-t,x_2,t)$ where $x = (x_1,x_2)$.    Setting $\xi= (\eta,\rho)$, we write   $|\xi| -\eta = \eta\sqrt{1 + \rho^2/\eta^2} -\eta = \eta \psi_{\mathsmaller{0}}(\rho/ \eta) $. 
Then,
\[ U(x,t)\cdot(\xi,|\xi|) = (x,t) \cdot (\eta,\rho, \eta \psi_{\mathsmaller{0}}(\rho/ \eta)).\]

By decomposition into finite angular sectors and rotation, we may assume $\widehat f$ is supported in a small conic neighborhood of $e_1$. 
Changing variables $(x,t) \to U(x,t)$, we have
\begin{equation}\label{eq}
\| e^{it\sqrt{-\Delta}} P_\lambda f \|_{L^q(\mathbb R^2\times I;\omega )} 
 \le C \| T_\lambda^{\psi_{\mathsmaller{0}}} f \|_{L^q(\mathbb R^2 \times I; \omega')} ,
\end{equation}
where $\omega'(x,t):=\omega( U(x,t))$ and $\widehat f$ is supported on
{$4^{-1}\lambda \le \eta \le 4\lambda$} 
and $|\rho/\eta| \le 1$.
Since $U$ is an {invertible linear map}, it is clear that $C^{-1}\omega' \in \Omega(\alpha)$ for some $C>0$.
Therefore, \eqref{TRg} holds for the same $\gamma$ and $p,q$ as in Proposition \ref{trg}.

The following lemma shows that estimates  relative to $\alpha$-dimensional measures can be deduced from weighted estimates.
See \cite[Section 3]{Mattila} for example.

\begin{lem}\label{convert}
Let $0<\alpha \le 3$ and $\nu \in \mathfrak C^{3}(\alpha)$. Set  $\varphi_\lambda =\lambda^3\varphi(\lambda \cdot)$ for  a Schwartz function $\varphi  $ such that
$\widehat{\varphi}=1$ on $ \mathbb B^3 (0,4) $ and $\supp \widehat{\varphi} \subset \mathbb B^3(0,8) $. There is a constant $C_1=C_1(\nu)>0$ such that $C_1^{-1} |\varphi_\lambda| \ast \nu \in \Omega(\alpha)$.
\end{lem}
\begin{proof}[Proof of Lemma \ref{convert}]
Let us define $  \nu_\lambda :=  |\varphi_\lambda|\ast \nu  $.
By rapid decay of $\varphi$,  we have
\begin{equation*} 
\nu_\lambda (z) \le 
C_N \lambda^3 \sum_{\ell\ge 0} 2^{-N\ell} \nu\big( \mathbb B^3(z ,2^\ell \lambda^{-1} )\big)    
\end{equation*}
for any $N>0$. It suffices to show $ \int_{\mathbb B^3(z_\circ ,r)} \nu_\lambda(y) dy 
\le   C_1 r^\alpha$ for a constant $C_1$.

Let us consider two cases $r \le \lambda^{-1}$ and $r >\lambda^{-1}$, separately.
If $r \le \lambda^{-1}$, using $\|\nu_\lambda \|_\infty \le  C \langle \nu \rangle_\alpha \lambda^{3-\alpha} $, we see $ \int_{\mathbb B^3(z_\circ ,r)} \nu_\lambda(z) dz 
\lesssim \langle \nu \rangle_\alpha \lambda^{3-\alpha}   r^3\le C_N'\langle \nu \rangle_\alpha r^\alpha$.   
To handle the case  $r>\lambda^{-1}$,  we use
\[
\int_{\mathbb B^3(z_\circ,r)} \nu\big( \mathbb B^3( z ,2^\ell \lambda^{-1} )\big)  dz \le \int_{\mathbb B^3(z_\circ,(2^\ell+1)r)} \int \chi_{\mathbb B^3(z', 2^\ell \lambda^{-1})}(z) dz d\nu(z').  \]
Taking $N$ sufficiently large, we get
\begin{align*}
\int_{\mathbb B^3(z_\circ,r)} \nu_\lambda (z) dz 
\le   C_N \lambda^3 \sum_{\ell\ge 0} 2^{-N\ell} 
(2^\ell \lambda^{-1})^3 (2^\ell+1)^\alpha r^\alpha  \le  C_N''  r^\alpha  . 
\end{align*}
Therefore, taking $C_1= \max\{  C_N',C_N''\}$, we  have $C_1^{-1}\nu_\lambda \in \Omega(\alpha)$. 
\end{proof}

We are ready to prove Theorem \ref{smoothing} and Theorem \ref{LS-est}.
\begin{proof}[Proof of Theorem \ref{smoothing} and Theorem \ref{LS-est}]
By  Littlewood-Paley decomposition, it is enough to show 
\begin{equation}\label{EE}
 \| e^{it\sqrt{-\Delta}} P_\lambda f \|_{L^q(  \mathbb R^2 \times I;\nu )}    \le C \lambda^{\gamma} \|f\|_{p}, \quad \lambda \ge 1 .
\end{equation}
Let  $\varphi$ be the Schwartz function given in Lemma \ref{convert}.
Since
$\mathcal F_{x,t}(e^{it\sqrt{-\Delta}} P_\lambda f) = \mathcal F_{x,t}(e^{it\sqrt{-\Delta}} P_\lambda f )  \widehat{\varphi_\lambda}$,
we have $|e^{it\sqrt{-\Delta}} P_\lambda f |^q \lesssim   |e^{it\sqrt{-\Delta}} P_\lambda f |^q  \ast |\varphi_\lambda|$ by H\"older's inequality.
Thus it follows that
\begin{equation}\label{trgg}
 \| e^{it\sqrt{-\Delta}} P_\lambda f \|_{L^q(  \mathbb R^2 \times I;\nu )}    \le C  \|e^{it\sqrt{-\Delta}} P_\lambda f \|_{L^q(\mathbb R^2 \times I;|\varphi_\lambda|\ast \nu )}  .
\end{equation}	
By  Lemma \ref{convert}, $C^{-1} |\varphi_\lambda| \ast \nu \in \Omega(\alpha)$ for some $C>0$. Since we are assuming  $\widehat f$ is supported in a small conic neighborhood of $e_1$,   using  \eqref{trgg} and \eqref{eq}, we see that the estimate \eqref{TR-est} implies \eqref{EE}
for the same $p,q$ and $\gamma$ as in  Proposition \ref{trg}  and Proposition \ref{trg22}.   
Therefore, Proposition \ref{trg}  and Proposition \ref{trg22}
give  Theorem \ref{smoothing} and Theorem \ref{LS-est}, respectively. 
\end{proof}

Next, we  show Theorem \ref{thm:circular-avr} and Corollary \ref{thm:circular}. 
 	
\subsection{Estimates for the circular average}
Recall the average $\mathcal A$ in $\mathbb R^d$, $d\ge 2$. 
Let us set   $P_0  = \sum_{j \le 0} P_{2^j} $. The kernel 
$\mathcal K(\cdot, t)$ of $f\to \mathcal AP_0f(\cdot, t)$ is rapidly decaying, i.e.,  $|\mathcal K(\cdot, t)|\le C(1+|x|)^{-N}$ for any $N$. 
By Schur's test, one can easily see   $\| \mathcal A P_0 f \|_{L^q(\mathbb R^2 \times I; d\nu)} \le C \|f\|_{p}$ for $1\le p \le q$. Therefore, 
it is sufficient to consider the contribution from   $\sum_{j\ge 1} \mathcal AP_{2^j}f$.

Let $d\sigma_t$ denote the normalized spherical measure. Then it is well known that
\Be
\label{asym}
\widehat{d\sigma_t}(\xi) = e^{it|\xi|} a_+(t\xi) +e^{-i t|\xi|} a_-(t\xi) ,
\Ee where $a_{\pm}$ are smooth functions satisfying $|\partial_\xi^m a_{\pm}(\xi)| \le C_m (1+|\xi|)^{-\frac{d-1}{2} - |m|}$ for any $m$.   
So, it is sufficient to consider $e^{it|\xi|} a_+(t\xi)$ since the contribution from $e^{-i t|\xi|} a_-(t\xi)$ can be handled similarly by reflection $t\to -t$. 
We set 
\[
\mathfrak A f (x,t)  =  \int e^{i (  x\cdot \xi + t |\xi|)} (1 + |\xi|^2)^{-\frac{d-1}{4}} \widehat f(\xi)d\xi    
\]
and 
\[ \mathfrak A_j f= \mathfrak A P_{2^j} f.\]
Since $t\in I$, to prove  the  estimate \eqref{nu-avr-d}
it suffices to show 
$
\| \sum_{j\ge 1} \mathfrak A_j  f \|_{L^q(\mathbb R^d\times I; d\nu)} \lesssim \langle \nu \rangle_\alpha^{1/q} \|f\|_p.
$
 Thus, the matter is reduced to showing 
\begin{equation}
\label{wave0}
\big\| \mathfrak A_j f \big\|_{L^q(\mathbb R^d\times I; d\nu)} \lesssim \langle \nu \rangle_\alpha^{\frac 1q} 2^{-\epsilon j } \|f\|_{p}, \quad j\ge1, 
\end{equation} 
for some $\epsilon>0$.

Let  $\mathbf K_j(\cdot, t)$ denote the kernel of $f\to \mathfrak A_jf(\cdot, t)$. It is easy to see 
\begin{align}\label{ker-wave}
|  \mathbf K_j (x,t)| \lesssim  2^{j}\big( 1+2^j\big| |x|-|t| \big| \big)^{-N} 
\end{align}
for any $N\ge1$ (see, for example,  \cite{Lee}). Using the estimate we obtain 
\begin{align} 
\label{00}
\| \mathfrak A_j f \|_{L^\infty(\mathbb R^d \times I;d\nu)}  &\lesssim  2^{j} \|f\|_1,   \\
\label{04}
\| \mathfrak A_j f \|_{L^1(\mathbb R^d \times I;d\nu)}  &\lesssim \langle \nu \rangle_\alpha \min \big\{ 2^j, 2^{(d+1-\alpha) j} \big\} \|f\|_1.
\end{align}
In fact,  \eqref{00} is clear from \eqref{ker-wave}. 
 To show \eqref{04}, we use the fact that $ \mathbf K_j $ is essentially supported in $O(2^{-j})$ neighborhood of the truncated cone with height $\sim 1$. Since 
the neighborhood  can be covered by $O(2^{dj})$ balls of radius $2^{-j}$, using \eqref{omega} we see $\| \mathbf K_j \|_{L^1(d\nu)} \lesssim   \langle \nu \rangle_\alpha 2^{(d+1-\alpha)j}$. 
Besides, we have $| \mathbf K_j |\lesssim 2^j(1+|x|)^{-N}$ since $t\in I$. So, $\|\mathbf K_j\|_{L^1(d\nu)}\lesssim \langle \nu \rangle_\alpha \min\{ 2^j, 2^{(d+1-\alpha)j} \}$. We get  \eqref{04} by Fubini's theorem.

\begin{proof}[Proof of Theorem \ref{thm:circular-avr}] 
As discussed above, we  need only to show
  \eqref{wave0}   
  for some $\epsilon>0$ if $p,q$ satisfy  \eqref{pq}.   When $1<\alpha \le 2$, the desired estimate \eqref{nu-avr-d} follows from Theorem \ref{smoothing} since  
\Be 
\label{wave}
\mathfrak A_jf=2^{-\frac j2} e^{it\sqrt{-\Delta}}P_{2^j}\widetilde f
\Ee for some $\widetilde f$ satisfying $\|\widetilde f\|_p \lesssim \|f\|_p$.

For $2<\alpha \le 3$, we use \eqref{00}, \eqref{04},  and 
   \begin{align} \label{m26}
\| \mathfrak A_jf \|_{L^r(\mathbb R^2\times I; d\nu)}  &\lesssim \langle \nu \rangle_\alpha^{\frac 1r} 2^{ (\frac{ 3-\alpha}r-\frac12+\epsilon) j} \|f\|_r, \quad 2\le r \le 4 ,~ \epsilon>0 .
\end{align}
To show \eqref{m26}, we use the estimate $\| \mathfrak A_j f \|_{L^r(\mathbb R^2 \times I)} \lesssim  2^{(-\frac12+\epsilon)j} \|f\|_{L^r}$ for $\epsilon > 0$ and $2\le r\le 4$, which follows from 
interpolation  between the estimates for $r=2$ and $r=4$.  Indeed, the first is clear from Plancherel's theorem  and \eqref{wave}. The second follows by the sharp local smoothing estimate \eqref{LS-AAA}
and \eqref{wave} combined with the globalization lemma (see Lemma \ref{globtri}). Then,  modifying the argument in the proof of Proposition \ref{w-globtri}, we get the estimate \eqref{m26}.

When $2<\alpha \le 3$, $\mathcal P_2(\alpha)$ is the closed quadrangle with vertices  $(0,0)$,
\[  \quad  \mathbf P_1:=\Big( \frac{\alpha}{\alpha+3}, \frac1{\alpha+3}\Big),  \quad\mathbf P_2:=   \Big( \frac{\alpha-1}\alpha, \frac1\alpha\Big),  \quad \mathbf P_3:= \Big( \frac1{4-\alpha}, \frac1{4-\alpha}\Big).    \]
Interpolating the estimates \eqref{04} and \eqref{m26} with $r=2$,  we have 
\[ 
\|  \mathfrak A_jf \|_{L^q(\mathbb R^2\times I; d\nu)}  \lesssim  \langle \nu \rangle_\alpha^{\frac 1q} 2^{(\frac{4-\alpha}q-1)j} \|f\|_q
\] 
for $1 \le q \le 2$. Since $\|  \mathfrak A_jf \|_{L^\infty(\mathbb R^2\times I; d\nu)}  \le  C \|f\|_\infty$, this gives \eqref{wave0} for $4-\alpha< {p=q} <\infty$. 
Therefore, to complete the proof,   by interpolation  it is enough to show \eqref{wave0}  for $(1/p,1/q)$   arbitrarily close to $\mathbf P_1$ and $\mathbf P_2$. 
Interpolating  \eqref{00} and \eqref{m26} yields 
\[ 
\| \mathfrak A_jf  \|_{L^q(\mathbb R^2\times I; d\nu)}  \lesssim  \langle \nu \rangle_\alpha^{\frac 1q} 2^{(\frac3{2p}+\frac{3-2\alpha}{2q}-\frac12+\epsilon )j} \|f\|_p
\] 
for $\epsilon>0$ if $1- \frac3q \le \frac 1p \le 1- \frac1q$ {and $p\le q$}.    A simple computation shows that this gives \eqref{wave0} for $(1/p,1/q)$ contained in 
the closed quadrangle with vertices  $(1/4, 1/4),  \mathbf P_1, \mathbf P_2, $ and $(1/2, 1/2)$ excluding the closed  line segment $[\mathbf P_1, \mathbf P_2]$. 
\end{proof}

Corollary \ref{thm:circular} is a direct consequence of Theorem \ref{thm:circular-avr} and the following lemma
combined with the globalization argument (e.g., Lemma \ref{globtri}).
For $0 < k \le d+1$, let $V$ be a $k$-dimensional subspace in $\mathbb R^{d+1}$.
We write $(x,t)=(z_1,z_2)$ where $z_1 \in V$, $z_2 \in V^\perp$.

\begin{lem}
\label{sup}
Let $V$ be a $k$-dimensional subspace in $\mathbb R^{d+1}$ for $k=1, \dots, d+1$.  
Let $\mu$ be an $\alpha$-dimensional measure on $V$  for some   $0<\alpha \le k$. 
If 
\eqref{nu-avr-d} holds with a uniform constant $C$, independent of $\nu \in \mathfrak C^{d+1}(\alpha)$,
then 
\[
\big\| \sup_{z_2\in V^\perp} | \chi_{\mathbb B^{d}(0,1)\times I}\mathcal Af(\cdot, z_2)| \big\|_{L^q(V, d\mu)} \le C \langle \mu \rangle_\alpha^{\frac 1q} \|f\|_p.
\]
\end{lem}

By Lemma \ref{sup}  and  the globalization argument, $L^p$--$L^q(d\mu)$ estimate for $f\to \sup_{t \in I} |\mathcal Af(\cdot,t)|$ can be deduced from 
\eqref{nu-avr-d} by taking $V=\mathbb R^d \times\{0\}$ and $\nu \in \mathfrak C^{d+1}(d)$.

\begin{proof}
By the Kolmogorov-Seliverstov-Plessner linearization it suffices to show 
\begin{equation}\label{suptR1} 
\|  \chi_{\mathbb B^d(0,1)\times I} \mathcal A f (\cdot\,, \mathbf z_2(\cdot))      \|_{L^q (d\mu)} \leq C \langle \mu \rangle_\alpha^{\frac1q}  \|f\|_{p} 
\end{equation} 
for any measurable function $\mathbf z_2: V \to V^\perp$ with $C$ independent of $\mathbf z_2$.  For any compactly supported continuous function $F$, we define a linear functional $\ell$ on $\mathrm C_c(\mathbb R^{d+1})$ by  
$ \ell(F) = \int F(z_1, \mathbf z_2(z_1)) d\mu(z_1)$.  
Then, by the Riesz representation theorem,  there is a unique Radon measure $\nu$ on $\mathbb R^{d+1}$ such that  
\[ \ell(F) = \int F(x,t) d\nu(x,t) = \int F(z_1, \mathbf z_2(z_1)) d\mu(z_1) .    \] 
Since  $\nu \big(\mathbb B^{d+1}((y_1,y_2),r) \big) \le  \int \chi_{\mathbb B^k(y_1,r)}(z_1)\,d\mu(z_1) \le \langle \mu\rangle_\alpha  r^\alpha$, the measure $\nu$ belongs to $\mathfrak C^{d+1}(\alpha )$ and  $\langle \nu \rangle_\alpha \le \langle \mu \rangle_\alpha$.
Therefore, $\|  \chi_{\mathbb B^d(0,1)\times I}\mathcal A f (\cdot\,, \mathbf z_2(\cdot))      \|_{L^q (d\mu)}   = \| \chi_{\mathbb B^d(0,1)\times I}\mathcal A f\|_{L^q (d\nu)}$, and  \eqref{suptR1} follows by \eqref{nu-avr-d}.
\end{proof}

\subsection{Estimate for the spherical average}
We close this section with results in higher dimensions.
As mentioned in Section \ref{Sec:1}, we can obtain some estimates for the spherical average from  
the fractal Strichartz estimates for the wave equation in $\mathbb R^{d+1}$.  
 
\begin{thm}[{\cite[Theorem 3.1]{CHL}}]\label{3d-frac}
Let $d\ge3$, $2\le q\le \infty$,  and $1<\alpha \le d+1$. Suppose  $\nu \in \mathfrak C^{d+1}(\alpha)$. 
Then, we have  
\begin{equation}\label{CHLop}
\| e^{it\sqrt{-\Delta}} f\|_{L^q(\mathbb R^d \times I;d\nu )}
\lesssim \langle\nu\rangle_\alpha^{\frac 1q}  \|f\|_{H^s}
\end{equation} 
holds  with  $s> s(\alpha,q,d)$, where  
\[
 s(\alpha,q,d)  = \begin{cases} 
\,  \max \big\{ \frac d2-\frac \alpha q , ~   \frac{d+1}4+\frac{1-\alpha}{2q}, ~     \frac{3d+1-2\alpha}8 \big\}, ~ \hfill   \quad 1 <\alpha \le d,\phantom{111} \\[1ex]
\,  \max \big\{ \frac d2-\frac \alpha q , ~   \frac{d+1}4+\frac{d+1-2\alpha}{2q}, ~     \frac{d+1- \alpha}2 \big\}, ~ \hfill  \quad  d <\alpha \le d+1.
\end{cases}
\]
\end{thm} 

In fact, the estimate  in \cite[Theorem 3.1]{CHL} is a local one. However, it can be easily extended  to a global estimate \eqref{CHLop} by the argument in Lemma \ref{globtri}.
The regularity exponents in Theorem \ref{3d-frac} are sharp when $d=3$ or $\alpha \in [1, \frac{d-1}2] \cup [d,d+1]$ in higher dimensions.

Using
\eqref{00} and \eqref{04}, we obtain an analogue of Theorem \ref{thm:circular-avr} for the spherical averages.
To state the result, let us set
\[ 
{p_d}(\alpha)= 
\begin{cases} 
\, \frac{\alpha+1}{\alpha} ,~  \hfill \quad 1<\alpha \le \frac{d-1}{2},  \\[.8ex]  
\, \frac{d+3+2\alpha}{d-1+ 2\alpha} ,~  \hfill \quad \frac{d-1}{2}<\alpha \le d.
\end{cases}
\]

\begin{thm}\label{maxthm-in-general}
Let $d\ge3$ and $1 < \alpha \le d+1$. Suppose $\nu \in \mathfrak C^{d+1}(\alpha)$. 
\begin{enumerate}[label={\rm{(}$\roman*$\rm{)}}, leftmargin=.7cm]
\item If $1<\alpha \le d$, then
\eqref{nu-avr-d} holds for $p,q$ satisfying $p > p_d(\alpha)$, 
\begin{equation}\label{pqx}
\begin{aligned}
q^{-1} &\le p^{-1} <\alpha q^{-1},\\   
(d+1)\inv p & < d-1 + (\alpha-1) \inv q. 
\end{aligned}
\end{equation} 
\item  
If $d<\alpha \le d+1$, then \eqref{nu-avr-d} holds for $p,q$ satisfying 
\begin{align*}
q^{-1} &\le p^{-1} <\alpha q^{-1}, \\
(d+1)p^{-1} &<d-1+(2\alpha-d-1)q^{-1},\\
dp^{-1} &<d-1+(\alpha-d)q^{-1}.
\end{align*}
\end{enumerate}
\end{thm}

When $d<\alpha\le d+1$, the results  are sharp (see Proposition \ref{prop:nec}). 
We also expect  the result to be sharp when  $d=3$, but  the necessity of the condition $p >p_3(\alpha)$ for $1<\alpha\le 3$ is verified only for $\alpha=2,3$ (see \ref{nnece4} and \ref{nnece4-1} in Proposition \ref{prop:nec}).  
Using the recent result due to Harris \cite{Harris2}, which improves the exponent $s$ in \eqref{CHLop} for $q =2$ and $\frac{d+1}{2} <\alpha < d$, 
$d\ge 4$, one can obtain a slightly better result.

\begin{proof}[Proof of Theorem \ref{maxthm-in-general}]
As in the proof of Theorem \ref{thm:circular-avr},  
it suffices to show  \eqref{wave0} for some $\epsilon>0$.   
Note $
\mathfrak A_jf=2^{-\frac {d-1}2j} e^{it\sqrt{-\Delta}}P_{2^j}\widetilde f
$
for some $\widetilde f$ satisfying $\|\widetilde f\|_p \lesssim \|f\|_p$ (cf. \eqref{wave}). 
By \eqref{CHLop}, we have 
\begin{equation}\label{2q-est}
\| \mathfrak A_j f \|_{L^{q}(\mathbb R^d \times I;d\nu)}   \lesssim \langle \nu \rangle_\alpha^{\frac 1{q}} 2^{(s- \frac{d-1}2)j} \|f\|_2 ,
\quad q \ge2
\end{equation}
for  $s >s(\alpha,q,d)$.

We first consider the case $1<\alpha \le d$.
We interpolate  \eqref{2q-est} and  \eqref{00}, and then  \eqref{2q-est} and  the trivial estimate  $\| \mathfrak A_j f \|_{L^{\infty}(\mathbb R^d \times I;d\nu)} \lesssim  \|f\|_\infty$. 
Consequently,  for  $  p, q $ satisfying  $\inv q \le \inv p \le 1-\inv q$  we have
\begin{equation}\label{pqpq}
\| \mathfrak A_j f \|_{L^{q}(\mathbb R^d \times I;d\nu)}   \lesssim \langle \nu \rangle_\alpha^{\frac 1{q}} 2^{\tilde s j} \|f\|_p
\end{equation}
if 
\[
\tilde s > \max  \Big\{ \frac 1p-\frac \alpha q , ~\,   \frac{d+1}{2p}+\frac{1-\alpha}{2q} - \frac{d-1}{2}, ~\,    \frac{d+3+2\alpha}{4p} - \frac{d-1+2\alpha}{4} \Big\} .
\]
Therefore,  \eqref{wave0} holds for some $\epsilon>0$ if $\tilde s <0$ i.e., $p,q$ satisfy \eqref{pqx},  $p > (d+3+2\alpha)/(d-1+2\alpha)$, and $\inv p \le 1 -\inv q$.

On the other hand, if $p,q$ satisfy  $1 -\inv q \le \inv p \le \inv q$, then we have  \eqref{pqpq} for 
\[
\tilde s > \max \Big\{ \frac{\alpha+1}{p} - \alpha, ~  \,  \frac{d+3+2\alpha}{4p} - \frac{d-1+2\alpha}{4} \Big\} , 
\]
which is a consequence of interpolation between \eqref{00}, \eqref{04}, and \eqref{2q-est} with $q=2$ and $s>s(\alpha,2,d) =  \max \big\{ \frac {d-\alpha}{2} , ~\frac{3d+1-2\alpha}8 \big\}$.
So, \eqref{wave0} holds for $p,q$ satisfying $1 -\inv q \le \inv p \le \inv q$ and $p>p_d(\alpha)$. 
Combining the results for the cases  $\inv q \le \inv p \le 1-\inv q$  and  $1 -\inv q \le \inv p \le \inv q$,  we have \eqref{wave0} for $ p,q $ satisfying \eqref{pqx} and $p>p_d(\alpha)$. 
This proves the first part of Theorem \ref{maxthm-in-general} (when $1<\alpha \le d$).

We next consider the case $d<\alpha \le d+1$.
We need to show \eqref{wave0} holds for $(1/p,1/q)$ contained in  the union of the open  quadrangle  with vertices $(0,0)$, 
\begin{equation*}
\scalebox{1}{$  \mathbf Q_1:=\Big( \frac{(d-1)\alpha}{(d-1)\alpha+(d+1)}, \frac{d-1}{(d-1)\alpha+(d+1)}\Big),~  \ \mathbf Q_2:=   \Big( \frac{\alpha-1}\alpha, \frac1\alpha\Big),~  \ \mathbf Q_3:= \Big( \frac{d-1}{2d-\alpha}, \frac{d-1}{2d-\alpha}\Big), $   }
\end{equation*} 
and the line segment $[(0,0), \mathbf Q_3)$. 
Interpolating \eqref{2q-est}, \eqref{00}, and the trivial $L^\infty$--$L^\infty(\mathbb R^d \times I;d\nu)$ estimate, we obtain \eqref{pqpq}  for 
$q^{-1} \le p^{-1} \le 1-q^{-1}$ if 
\[
\tilde s > \max  \Big\{ \frac 1p-\frac \alpha q , ~\,   \frac{d+1}{2p}+\frac{d+1-2\alpha}{2q} - \frac{d-1}{2}, ~\,    \frac{\alpha}{p} - (\alpha-1) \Big\} .
\]
Thus,  \eqref{wave0} holds for some $\epsilon>0$ if $\tilde s <0$ i.e.,  $(\inv p,\inv q)$ is contained in  the interior of the quadrangle with vertices $(0,0)$, $\mathbf Q_1$, $\mathbf Q_2$, and $(1/2,1/2)$. 

Finally, by interpolation we  need only to obtain \eqref{wave0} for $(1/p,1/q)\in [(1/2,1/2), \mathbf Q_3)$. 
To see this, we interpolate two estimates  \eqref{04} and \eqref{2q-est} with $q=2$ and $s > (d+1-\alpha)/2$ to obtain \eqref{pqpq} for $1\le p=q\le 2$ and $\tilde s > \frac{2d-\alpha}{p} -(d-1)$. 
Thus,  \eqref{wave0} holds for $(\frac 1p, \frac 1q) \in [(1/2,1/2), \mathbf Q_3)$, as desired.
\end{proof}

Similarly as in Corollary \ref{thm:circular}, 
$L^p$-$L^q(d\mu)$ estimate for the spherical maximal function holds for the range of $p, q$ as in Theorem \ref{maxthm-in-general}.

\begin{cor}\label{3dm}
Let $d\ge3$, $1<\alpha \le d$ and $\mu \in \mathfrak C^{d}(\alpha)$. For $p,q$ satisfying $p > p_d(\alpha)$ and \eqref{pqx}, we have
\begin{align}\label{sup-high}
\big\|   \sup_{t \in I} |\mathcal A  f (\cdot,t)| \big\|_{L^q(d\mu)}   \le C \langle \mu\rangle_\alpha^{\frac 1q} \|f \|_{L^p(\mathbb R^d)} .
\end{align}
\end{cor}

\section{Necessary conditions on  $p,q$ for \eqref{nu-avr-d}} 
\label{Sec:4}

In this section, we obtain necessary conditions on $p,q,\alpha$ for \eqref{nu-avr-d} to hold.  
Obviously, it is enough to consider the local estimate 
\Be\label{nu-avr-loc} 
\| \mathcal Af \|_{L^q ( \mathbb B^d(0,1) \times I ; d\nu)} \le C\langle\nu\rangle_\alpha^{\frac1q} \|f \|_{L^p (\mathbb R^d)}. 
\Ee 
This is done by taking specific $\alpha$ dimensional measures and particular functions. 
A modification of these examples also gives   necessary conditions for the maximal estimate \eqref{sup-high} (also see \eqref{est-circ}) when $1<\alpha \le d$.

\begin{prop}\label{prop:nec}
Let $d \ge 2$ and $1<\alpha \le d+1$. 
Suppose \eqref{nu-avr-loc} holds for  $\nu \in \mathfrak C^{d+1}(\alpha)$.  Then, 
\begin{enumerate}[leftmargin=20pt, align=left, label={\rm{(}$\roman*$\rm{)} }]
\item \label{nnece1} $\qquad \quad \ \inv p \le \alpha \inv q$,
\\
\item \label{nnece1-1} $(d+1)\inv p \le \begin{cases} \, d-1+ (\alpha-1) \inv q, ~\hfill \quad 1<\alpha \le d, \phantom{111} \\ \,  d-1+  (2\alpha-d-1)\inv q , ~ \hfill \quad d<\alpha \le d+1 , \end{cases}$ 
\\[4pt]
\item \label{nnece1-2} $\qquad d\inv p  \le  \begin{cases} \, d-1   , ~ \hfill \quad  1<\alpha \le d , \phantom{111} \\ \,  d-1 +  (\alpha-d)\inv q , ~ \hfill \quad  d<\alpha \le d+1 . \end{cases}$ 
\end{enumerate}
Furthermore, if $d\ge3$ and $k <\alpha \le k+1$ for an integer $k \in [1, d-1]$, then the following are necessary for \eqref{nu-avr-loc} to hold: 
\begin{enumerate}[resume*] 
\item \label{nnece4} 
$\begin{aligned} 
\hphantom{g}
\quad (d+k) p^{-1}  &\le  d+k-2   + (\alpha-k)  q^{-1},
\\  p^{-1}  &\le \frac{d+k-1}{d+k+1},
\end{aligned}$
 \\[10pt]
\item \label{nnece4-1} 
$\begin{aligned} \qquad (k+1) p^{-1}  &\le   k  + (\alpha-k) q^{-1}, \\  
p^{-1}  &\le \frac{k+1}{k+2}.
\end{aligned}$
\end{enumerate}
When  $d=2$ and  $1<\alpha \le 3$, \eqref{nu-avr-loc} holds only if 
\begin{enumerate}[resume*]  
\item \label{nnece5} $p \ge 4-\alpha$. 
\end{enumerate}
\end{prop}

It seems to be natural to expect that   \eqref{nu-avr-loc} holds for  $\nu \in \mathfrak C^{d+1}(\alpha)$   only if, instead of \ref{nnece4} and \ref{nnece4-1},
\begin{align}\label{nnece3}
p^{-1} \le
\begin{cases}
\frac{d+\alpha-2}{d+\alpha}, ~ & ~   \hfill    d-2<\alpha \le d, \\[.8ex]
\frac{\alpha}{\alpha+1}, ~ & ~   \hfill 1<\alpha \le d-2.
\end{cases}
\end{align}
When $\alpha=k+1$ ($k=1,\dots,d-1$), \eqref{nnece3} matches with the second conditions in \ref{nnece4} and \ref{nnece4-1}.
In view of the global estimate \eqref{nu-avr-d} with $d=2$,  the condition \ref{nnece5} is redundant since \eqref{nu-avr-d} holds only if $p\le q$. In fact, the conditions \ref{nnece1}, \ref{nnece1-1}, \ref{nnece1-2} combined with $p \le q$ show  that \eqref{nu-avr-d} with $d=2$ fails unless \ref{nnece5} holds.

\subsection*{Proof of Proposition \ref{prop:nec}}
Let us assume that \eqref{nu-avr-loc} holds for all $\nu \in \mathfrak C^{d+1}(\alpha)$ with $1<\alpha\le d+1$. 
We show the necessity of the conditions \ref{nnece1}--\ref{nnece5}
by considering specific measures contained in  $\mathfrak C^{d+1}(\alpha)$ and functions which show 
the failure of \eqref{nu-avr-loc} unless the conditions hold.  

In what follows,  $0<\delta \ll 1$ and $c$ is a small positive constant.

\begin{proof}[Proof of \ref{nnece1}] 
Let $N_\delta=\{ x: ||x|-1|<\delta\}$ and set $f = \chi_{N_\delta}$. Clearly,
$ \mathcal A f(x,t )  \gtrsim 1$ if $( x,t ) \in    \mathbb B^d(0, c \delta)\times [1, 1+ c\delta]$. 
Let $d\nu(x,t) = \psi(x,t) |x|^{\alpha - d-1} dx dt$ for $\psi \in \mathrm C_0^\infty(\mathbb B^{d+1}(0,3))$.
It is easy to see $ \langle \nu \rangle_\alpha \lesssim 1$. Then \eqref{nu-avr-loc} implies $\delta^{\frac\alpha q} \lesssim \delta^{\frac 1p}$.
Thus, letting $\delta\to 0$, we get $1/p \le \alpha/ q$. 
\end{proof}

\begin{proof}[Proof of \ref{nnece1-1}] 
Let us set 
\[  V_\delta=\{(x_1,x',t) \in \mathbb R \times \mathbb R^{d-1} \times [1,2]: |x_1-t|\le c \delta^2, |x'|\le  c \delta\}.\]
We also define a measure $\nu$ by 
\[
d\nu (x,t) = \begin{cases}
\, \delta^{\alpha-d-2} \chi_{V_\delta}\,dxdt, ~\hfill \quad 1<\alpha \le d, \phantom{111} \\[.7ex]
\, \delta^{2(\alpha-d-1)} \chi_{V_\delta}\,dxdt,~ \hfill \quad  d<\alpha \le d+1 .
\end{cases}
\]
One can easily show $\langle \nu \rangle_\alpha \lesssim 1$ for $1<\alpha \le d+1$. 
Indeed, when $1<\alpha \le d$,    
$\nu(\mathbb B^{d+1}(z,r))\lesssim \delta^{\alpha-d-2}r^{d+1}$ for $0<r<\delta^2$, 
$\nu(\mathbb B^{d+1}(z,r))\lesssim \delta^{\alpha-d} r^{d}$ for $\delta^2<r<\delta$, and
$\nu(\mathbb B^{d+1}(z,r))\lesssim \delta^{\alpha-1}r$ for $\delta <r$.
So we have $\nu(\mathbb B^{d+1}(z,r)) \lesssim r^\alpha$ in each case.
The case  $d<\alpha \le d+1$ can be handled similarly.

We consider $f= \chi_{U_\delta}$ for $ U_\delta= [0,\delta^2] \times [0,\delta]^{d-1}$.   
Then, we have $\mathcal Af(x,t) \gtrsim \delta^{d-1}$  if $(x,t) \in V_\delta$.
Note $\nu (V_\delta) \sim  \delta^{\alpha-1}$ for $1<\alpha \le d$, and $\nu(V_\delta) \sim \delta^{2\alpha-d-1}$ for $d<\alpha \le d+1$.
Since $\|f\|_p =\delta^{\frac{d+1}{p}}$, \eqref{nu-avr-loc} implies $\delta^{d-1} \delta^{\frac{\alpha-1}{q}}
\lesssim \delta^{\frac{d+1}{p}}$ for $1<\alpha \le d$, which gives \ref{nnece1-1} for $1<\alpha \le d$.
When $d<\alpha\le d+1$, we get \ref{nnece1-1} in the same manner. 
\end{proof}

\begin{proof}[Proof of \ref{nnece1-2}]
Let us take $f = \chi_{\mathbb B^d(0,\delta)}$ so that $\| f\|_p \lesssim \delta^{d/ p}$.
Then we have $\mathcal Af \gtrsim \delta^{d-1} \chi_{S_\delta}$ where
$ S_\delta = \{ (x,t):  ||x|-t| \le c\delta,~ t \in [1,2] \}$  for a small $c>0$. 
Let 
\[
d\nu(x,t)= \begin{cases} \delta^{-1} \chi_{S_\delta} dx dt   , ~ \hfill \quad 1 < \alpha \le d , \phantom{111}  \\[.7ex]
\,  \delta^{\alpha-d -1} \chi_{S_\delta} dx dt, ~ \hfill \quad d<\alpha \le d+1.
\end{cases}
\]
It is not difficult to see $\langle \nu \rangle_\alpha \lesssim 1$. 
When $d<\alpha \le d+1$, $\nu(\mathbb B^{d+1}(z,r))\lesssim \delta^{\alpha-d-1}r^{d+1}$ for $0<r \le \delta$, and $\nu(\mathbb B^{d+1}(z,r))\lesssim \delta^{\alpha-d}r^d$  for $\delta<r$. The case $1<\alpha \le d$ can be treated similarly and we omit the detail.
Since $|S_\delta| \sim \delta$, we have $\nu( S_\delta) \sim 1$ for $1<\alpha \le d $, and  $\nu( S_\delta) \sim \delta^{\alpha-d}$ for $d<\alpha \le d+1$. So, 
\eqref{nu-avr-loc} implies $\delta^{d-1}\lesssim \delta^{d/ p}$ for $1<\alpha \le d $, and  $\delta^{d-1+(\alpha-d)/q}\lesssim \delta^{d/ p}$ for  $d<\alpha \le d+1$. 
Letting $\delta\to 0$ yields  \ref{nnece1-2}.
\end{proof}

\begin{proof}[Proof of \ref{nnece4}] We 
modify  the example in the proof of \ref{nnece1-1}.
For a given $\alpha \in (1,d]$, let $k \in [1,d-1]$ be an integer such that  $k<\alpha \le k+1$.
Let  $U_\delta^k= [0,\delta^2]^k \times [0,\delta]^{d-k}$. We have
\Be
\label{lower} \mathcal A  \chi_{U_\delta^k}  \gtrsim \delta^{2(k-1)+ d-k} \chi_{V_\delta^k} ,
\Ee
where 
\[
V_\delta^k
:=  \big\{ (x',x'',t) \in \mathbb R^k \times \mathbb R^{d-k} \times [1,2] : \quad~  ||x'|-t| \le c \delta^2, \quad~
|x''| \le c \delta \big\}.
\]

Let us set $d\nu=\delta^{\alpha-d-2}  \chi_{V_\delta^k} dxdt$. In the same manner as before  it is easy to show  $\langle \nu \rangle_\alpha \lesssim 1$. 
 Since  $\nu(V_\delta^k) \sim \delta^{\alpha-k}$ and $\|  \chi_{U_\delta^k} \|_p \lesssim \delta^{(d+k)/{p}}$, \eqref{nu-avr-loc} and \eqref{lower} yield  $\delta^{d+k-2}\delta^{(\alpha-k)/{q}} \lesssim  \delta^{(d+k)/{p}}$. This gives the first part of \ref{nnece4}.

For the second part of \ref{nnece4}, taking  $f= \chi_{U_{\delta}^{k+1}} $ instead of $ \chi_{U_{\delta}^{k}} $, we have
 $ \mathcal A f  \gtrsim \delta^{d+k-1} \chi_{V_\delta^{k+1}} $.
We consider $d\nu = \delta^{k-d-1}  \chi_{V_\delta^{k+1}} dxdt$. Then, similarly as before, one can easily see  $\langle  \nu \rangle_\alpha \lesssim 1$ and $\nu(V_\delta^{k+1}) \gtrsim1$. 
 Since $\| \chi_{U_\delta^{k+1}} \|_p \lesssim \delta^{(d+k+1)/{p}}$, \eqref{nu-avr-loc} implies $\delta^{d+k-1}\lesssim \delta^{(d+k+1)/{p}}$.
This shows the second part of \ref{nnece4}.
\end{proof}

\begin{proof}[Proof of \ref{nnece4-1}]
We show \ref{nnece4-1} by combining  the examples considered while proving  \ref{nnece1} and \ref{nnece1-2}. 
Let $U_k$ be the $\delta$-neighborhood of the set $\{ (0,\tilde z) \in \mathbb R^k \times \mathbb R^{d-k} : |\tilde z|=1\}$.
We set $f= \chi_{U_k}$ so that $\|f\|_p \lesssim \delta^{\frac {k+1}p}$.
Parametrizing $\mathbb S^{d-1}$, we have 
\[
 \mathcal A f(x,t )  \ge    \int_{|y| \le 1/2} \chi_{U_k} 
\big( x- t(y ,\sqrt{1-|y|^2}) \big) \,dy   . 
\]
Denoting $x=(x',x'',x_d) \in \mathbb R^k \times \mathbb R^{d-k-1}\times \mathbb R$, we set 
\[
V_k = \big\{ (x,t) : \big| |x'|-\sqrt{t^2-1} \big| \le c \delta ,~ |x''|,  |x_d|\le c \delta, ~ t \in [3/2,2] \big\}.
\]

We write $y  = (y',y'')\in \mathbb R^k \times \mathbb R^{d-k-1}$. 
Then, we  claim 
\begin{equation}\label{cl} x-t(y',y'',\sqrt{1-|y|^2}) \in U_k \end{equation}
for $t\in [3/2,2]$ and $|y''|\le 1/2 $
 if $(x,t)  \in V_k$ and  $|x'-ty'|\le c\delta$. 
Since $|t^{-1}x'- y'|\le c\delta$  and $| |x'| - \sqrt{t^2-1}|\le  c\delta$, it is clear that $||y'|-  t^{-1}{\sqrt{t^2-1}}|<  3c \delta$.
This gives $t|(y'',\sqrt{1-|y|^2})|=t\sqrt{1-|y'|^2}=1+O(c\delta)$.
Thus,  for  $|x''|, |x_d| \le c \delta$ we have
\[|(x'',x_d)-t(y'', \sqrt{1-|y'|^2-|y''|^2} )|\in (1-2^{-1}\delta, 1+2^{-1}\delta)
\]
if $c$ is small enough. This proves \eqref{cl}.  
Since $y' = x'/t+O(c\delta) $ and  $|y''| \lesssim 1$, \eqref{cl} implies 
\[  \mathcal A f   \gtrsim \delta^k  \chi_{V_k} .\]
Let $d\nu=\delta^{\alpha-d-1} \chi_{V_k}dxdt$. Then  we have $\nu(V_k) \gtrsim \delta^{\alpha-k}$, and $\langle \nu \rangle_\alpha \lesssim1$ since 
 $\nu(\mathbb B^{d+1}(z,r)) \lesssim \delta^{\alpha-d-1}r^{d+1}$ if $0<r \le \delta$
and $\nu(\mathbb B^{d+1}(z,r))\lesssim \delta^{\alpha-k}r^k$ if $\delta<r$.
So  \eqref{nu-avr-loc} gives  $\delta^k \delta^{\frac{\alpha-k}q} \lesssim \delta^{\frac{k+1}p}$.  Hence the first part of \ref{nnece4-1} follows if we take $\delta\to 0$.

To prove the second part of \ref{nnece4-1}, 
taking $f=\chi_{U_{k+1}}$, we have
$\|f\|_p \lesssim \delta^{\frac{k+2}p}$, and $\mathcal A  \chi_{U_{k+1} }  \gtrsim \delta^{k+1}$ whenever
$(x,t) \in V_{k+1}$.
Considering $d\nu = \delta^{k-d}\chi_{V_{k+1}}dxdt$, we have $\langle \nu \rangle_\alpha \lesssim 1$ and $\nu(V_{k+1}) \gtrsim 1$ since $\nu(\mathbb B^{d+1} (z,r) ) \lesssim \delta^{k-d}r^{d+1}$ for $0<r\le \delta$, and $\nu(\mathbb B^{d+1}(z,r)) \lesssim \delta r^k$ for $\delta<r$. Therefore, \eqref{nu-avr-loc} implies $\delta^{k+1} \lesssim \delta^{\frac{k+2}p}$ for $0<\delta \ll1$.
This gives the second part of \ref{nnece4-1}.
\end{proof}

\begin{proof}[Proof of \ref{nnece5}] Let us consider the case $1<\alpha \le 2$ first. 
Let $\{ \theta_\ell \}$ be a collection of $\sim C\delta^{1-\alpha}$ points 
on $\mathbb S^{1}$ which are separated by $\delta^{\alpha-1}$. For each $\theta_\ell$, we set 
\[
\mathrm T_\ell(\delta)=\big\{ (x,t) \in \mathbb R^2 \times [1,2]:  | |x| -t| \le c\delta^2,~  \big|x/|x| -\theta_\ell \big| \le c
\delta \big\}.
\] 
Let $\mathrm T_\delta=\cup_{\ell} \mathrm T_\ell(\delta)$. We consider 
$d\nu(x,t) = \delta^{\alpha-4} \chi_{\mathrm T_\delta}(x,t)\,dxdt$, 
then $\langle \nu \rangle_\alpha \lesssim 1$.
In fact, one can easily see
\begin{align*}
\nu(\mathbb B^3(z,r))\le
C\begin{cases}
\ \ \ \ \delta^{\alpha-4} r^3, \quad\quad & \quad \ \ 0<r\le \delta^2,\\
\ \ \ \ \delta^{\alpha-2}r^2, \quad\quad &\ \  \ \, \delta^2<r \le \delta,\\
\ \ \ \ \delta^{\alpha-1}r, \quad\quad & \ \ \  \  \ \delta <r \le \delta^{1-\alpha},\\
\delta^{\alpha-1}r  \times (r/ \delta^{\alpha-1} ), \quad\quad & \delta^{\alpha-1} < r < 1,\\
\delta^{\alpha-1}   \times (r/ \delta^{\alpha-1}), \quad\quad &\quad \ \ 1  < r .
\end{cases}
\end{align*}
Therefore, we have $\nu(\mathbb B^3(z,r)) \lesssim r^\alpha$.  

Let  $\mathrm R(\theta, \delta)$ be a rectangle of side length $\delta^2 \times \delta$ centered at the origin such that the short side is parallel to  $\theta\in \mathbb S^1$. 
Setting $\mathrm R_\delta := \cup_\ell \mathrm R(\theta_\ell, \delta)$, we have $\mathcal A \chi_{\mathrm R_\delta}  \gtrsim  \delta  \chi_{\mathrm T_\delta}$. 
Since $\|\chi_{\mathrm R_\delta} \|_p \lesssim \delta^{(4-\alpha)/p}$ and $\nu(\mathrm T_\delta) \gtrsim 1$, we get $\delta \lesssim \delta^{(4-\alpha)/p}$ which implies $p \ge 4-\alpha$.  

Now, we consider the case $2<\alpha \le 3$. Let $\{\omega_j \} \subset \mathbb R^2$ be a collection of $\sim C\delta^{2-\alpha}$ points  on the $x_1$-axis which are separated by $\delta^{\alpha-2 }$. 
For each $\omega_j$, we denote 
\[
\mathbf C_j(\delta) = \{ (x,t) \in \mathbb R^2 \times [1,2] : ||x-\omega_j|-t| \le c\delta, \quad   |x_2|\le 2^{-3}    \}.
\]
We set $\mathcal T_\delta=\cup_j \mathbf C_j(\delta)$ and 
$
d\nu(x,t)= \delta^{\alpha-3} \chi_{\mathcal T_\delta}(x,t)\,dxdt.
$
Then, it is easy to  see  $\langle \nu \rangle_\alpha \lesssim 1$ and $\nu(\mathcal T_\delta)\sim1$.
Indeed, 
\begin{align*}
\nu(\mathbb B^3(z,r)) \le C 
\begin{cases}
\ \ \ \ \delta^{\alpha-3}r^3, \quad &\quad  \ \ 0 <r\le\delta,\\
\ \ \ \ \delta^{\alpha-2}   r^2 \quad &\quad \ \  \delta<r<\delta^{ \alpha-2 },\\
\delta^{\alpha-2}r^2\times (r/\delta^{ \alpha-2}), \quad &\delta^{ \alpha-2 }< r<1,\\
\delta^{\alpha-2} \times (r/\delta^{ \alpha-2}), \quad &\quad \ \ 1 < r.
\end{cases}
\end{align*}
So, it follows that $\nu(\mathbb B^3(z,r)) \lesssim r^\alpha$.
Also we have
\[\mathcal A\chi_{\cup_j \mathbb B^2(\omega_j, c\delta)} (x,t)\gtrsim \delta, \quad (x,t) \in \mathcal T_\delta.\]  
Since $|\cup_j \mathbb B^2(\omega_j, c\delta)|\lesssim \delta^{4-\alpha} $, \eqref{nu-avr-loc} implies 
$\delta \lesssim \delta^{\frac{4-\alpha}p}$. Thus, we see  $p\ge 4-\alpha$.
\end{proof}

\subsection*{Acknowledgement} 
This work was supported by the NRF (Republic of Korea) grants  No. 2017R1C1B2002959 (Ham), No. 2019R1A6A3A01092525 (Ko), and No.  2021R1A2B5B02001786 (Lee).

\bibliographystyle{plain}

\end{document}